\documentclass[11pt]{amsart}
\usepackage{amsmath,amssymb,mathrsfs,color,amsthm}
\usepackage{bm}
\usepackage{array}
\usepackage{longtable}
\usepackage{tabularx}
\usepackage{booktabs}
\usepackage[title]{appendix}
\usepackage{amssymb}
\usepackage{graphicx}  
\usepackage{graphicx}  
\usepackage{caption}   
\usepackage{enumitem}
\usepackage{relsize}
\usepackage{graphicx}
\usepackage{xcolor}
\usepackage{hyperref}
\allowdisplaybreaks

\def\R{{\mathbb R}}
\def\P{{\mathbb P}}
\def\E{{\mathbb E}}

\def\L{\mathcal{L}}

\newtheorem{thm}{Theorem}[section]

\newtheorem{lem}[thm]{Lemma}
\newtheorem{prop}[thm]{Proposition}

\theoremstyle{definition}
\newtheorem{de}[thm]{Definition}
\theoremstyle{remark}
\newtheorem{rem}[thm]{Remark}

\newtheorem{ass}[thm]{\bfseries Assumption}
\numberwithin{equation}{section}

\newcommand{\I}{\mathcal{I}}
\newcommand{\J}{\mathcal{J}}
\newcommand{\rmd}{{\rm d}}
\newcommand{\rme}{{\rm e}}
\newcommand{\tr}{{\rm tr}}
\newcommand{\F}{{\mathcal{F}}}
\allowdisplaybreaks

\begin{document}

\title[Generalized gradient flow and GENERIC]{Nonlinear Vlasov-Fokker-Planck equations: From generalized Wasserstein gradient flow to GENERIC structure}

\author{Zhenxin Liu}
\address{Z. Liu: School of Mathematical Sciences, Dalian University of Technology, Dalian
116024, P. R. China}
\email{zxliu@dlut.edu.cn}

\author{Xuewei Wang}
\address{X. Wang (Corresponding author): School of Mathematical Sciences, Dalian University of Technology, Dalian
116024, P. R. China}
\email{Wangxueweii@163.com}

\date{April 10, 2026}



\begin{abstract}
We study the GENERIC (General Equation for Non-Equilibrium Reversible Irreversible Coupling) formulation of the nonlinear Vlasov-Fokker-Planck equation from the perspective of gradient flows along trajectories.
After pulling back the reversible component, the evolution can be recast as a generalized Wasserstein gradient flow.
The associated free energy functional consists of an entropy term and a conservative energy term, while the metric is induced by the Onsager operator.
This trajectory based viewpoint shows that the trajectorial rate of free energy dissipation captures the influence of the nonlinear term, an effect that is not directly apparent at the macroscopic level.
Finally, partial degeneracy yields a decomposition of the underlying metric space, which in turn enables the derivation of a partial HWI inequality.
\end{abstract}

\maketitle

\section{Introduction}

In \cite{JKO2}, Jordan, Kinderlehrer and Otto first established a Wasserstein gradient flow formulation of the Fokker-Planck equation through a time discretization scheme.
This variational approach has since been extended to a wide class of evolution equations.
To extract more refined information at the microscopic level, Karatzas et al. analyzed energy dissipation on a trajectory basis by means of an associated stochastic differential equation in \cite{KST}.
This trajectory based methodology has subsequently been adapted to nonlinear \cite{LW, TC} and degenerate \cite{KC} contexts.

However, not every Fokker-Planck equation admits the classical Wasserstein gradient flow structure.
In the partially degenerate setting, our objective is to identify an appropriate metric and to establish a corresponding Wasserstein-type gradient flow formulation.
In particular, we consider the following Vlasov-Fokker-Planck equation
\begin{equation}\label{VFPE}
\begin{aligned}
\partial_t f+v\nabla_x f-\dfrac{1}{m}\nabla_x U \nabla_v f-\dfrac{1}{m}\nabla_x (K*f)\nabla_v f=\dfrac{\gamma}{m}\nabla_v & (vf+\dfrac{k_B T_B}{m}\nabla_v f),
\end{aligned}
\end{equation}
where $(t,x,v) \in [0, +\infty)\times \R^d \times \R^d$ and the initial condition is given by $f_0$.
The smooth functions $U ,K : \R^d \rightarrow [0, +\infty)$ denote the prescribed confinement and interaction potentials, respectively.
Here $m, \gamma, k_B, T_B$ are fixed positive constants representing, respectively, the particle mass, the friction coefficient, the Boltzmann constant, and the bath temperature.
The convolution is defined by
$K*f(x):=\int_{\R^{2d}}K(x-x')f(x', v')\rmd x' \rmd v'$.
For convenience, we introduce the notation $\beta=\frac{1}{k_B T_B}$ and $\sigma=\sqrt{\frac{2\gamma}{\beta m^2}}$.

Unfortunately, the degenerate system need not admit a gradient flow formulation.
As shown in \cite{DPZ}, \eqref{VFPE} possesses a GENERIC structure in the sense that its dynamics decompose into reversible and dissipative components.
Formally, we have two functionals $\mathrm{E}$ and $\mathrm{S}$ such that
\begin{equation}\label{gen}
\partial_t f = \mathrm{L}(f) \rmd \mathrm{E}(f) + \mathrm{J}(f) \rmd \mathrm{S}(f),
\end{equation}
where $\mathrm{L}$ and $\mathrm{J}$ are Poisson and Onsager operators, respectively.
In \cite{KLMP}, Kraaij et al. derived a GENERIC formulation from a new fluctuation symmetry. Their approach extends the gradient flow paradigm by incorporating a Hamiltonian component while preserving the free energy as a Lyapunov function.
In \cite{MPZ}, Mielke et al. began with a Hamiltonian system and, by combining energy, entropy, and the Onsager operator, derived a GENERIC framework for the microscopic dynamics.
By contrast, we place greater emphasis on energy dissipation and develop a generalized Wasserstein gradient flow description at the trajectory level.
To this end, we analyze the mean-field Langevin equation
\begin{equation}\label{MFLE}
\begin{aligned}
\rmd X_t&=V_t\rmd t, \\
\rmd V_t&=-\dfrac{1}{m}\nabla U(X_t)\rmd t -\dfrac{1}{m}\nabla (K*\rho_t)(X_t) \rmd t -\dfrac{\gamma}{m} V_t \rmd t +\sigma \rmd B_t
\end{aligned}
\end{equation}
on $\R^d \times \R^d$. Here $\rho_t:={ \rm Law}(X_t, V_t)$ and $(B_t)_{t\geq 0}$ denotes a standard Brownian motion on a probability space $(\Omega, \mathcal{G}, \P)$ with normal filtration $(\mathcal{G}_t)_{t\geq 0}$.
The drift depends on the law of the process, and such equations are commonly referred to as McKean-Vlasov stochastic differential equations \cite{K}.

The free energy functional comprises an entropic contribution and a conservative energy term. It can be expressed as
\begin{equation}\label{fenergy}
\mathcal{F}(f):= \int_{\R^{2d}} f\ln f +\beta m f \check{h}_f \rmd x \rmd v,
\end{equation}
where $\check{h}_f(x,v):=\frac{1}{2}v^2 +\frac{1}{m}U(x)+\frac{1}{2m} K*f(x)$.
We pull back the dynamics along the flow $\Phi$ generated by the reversible component
and denote the resulting distribution and density by $\mu$ and $u$, respectively.
To avoid ambiguity, we write $(\tilde{x}, \tilde{v})^{\top}$ for the coordinates after the pullback of $(x, v)^{\top}$.
Other variables are treated analogously.
At this point, the energy functional
\begin{equation}\label{penergy}
\tilde{\F}_t(u):=\int_{\R^{2d}} u(\tilde{x}, \tilde{v}) \ln u(\tilde{x}, \tilde{v})+ \beta m u(\tilde{x},\tilde{v}) \check{h}_{f_t}(\Phi_t(\tilde{x}, \tilde{v}))  \rmd \tilde{x} \rmd \tilde{v}
\end{equation}
satisfies $\F(f_t) =\tilde{\F}_t(u_t)$.
Applying the semimartingale decomposition to the transformed system
yields a trajectorial identity for the rate of free energy dissipation of the form
\begin{equation*}
\begin{aligned}
\lim_{t \downarrow t_0} \dfrac{\mathbb{E}[\tilde{\theta}_{t}(\tilde{X}_{t}, \tilde{V}_{t}, \rho_{t})| \mathcal{G}_{t_0}] -\tilde{\theta}_{t_0}(\tilde{X}_{t_0}, \tilde{V}_{t_0}, \rho_{t_0}) }{t-t_0}
= \tilde{\mathcal{D}}_{t_0}(\tilde{X}_{t_0}, \tilde{V}_{t_0}), ~ \P-a.s.,
\end{aligned}
\end{equation*}
where the explicit expression of $\tilde{\mathcal{D}}_t$ is given in \eqref{D} and
\begin{equation}\label{energyp}
\tilde{\theta}_t(\tilde{X_t}, \tilde{V_t}, \rho_t):=\ln u_t(\tilde{X_t}, \tilde{V_t}) + \beta m \check{h}_{\rho_t}(\Phi_t(\tilde{X_t},\tilde{V_t}))
\end{equation}
denotes the energy process.
Moreover, we obtain the corresponding free energy dissipation
\begin{equation}\label{diss}
\begin{aligned}
\dfrac{\rmd}{\rmd t}\tilde{\F}_t(u_t)=& -\dfrac{1}{2} \int_{\R^{2d}} | \nabla (\ln u_t(\tilde{x}, \tilde{v})+ \beta m \check{h}_{f_t}(\Phi_t(\tilde{x}, \tilde{v})) ) D\Phi_{-t}(\Phi_t(\tilde{x}, \tilde{v}))G |^2 u_t(\tilde{x}, \tilde{v}) \rmd \tilde{x} \rmd \tilde{v} \\
=&:-\tilde{\I}(u_t) > -\infty,
\end{aligned}
\end{equation}
where $G$ is the diffusion coefficient of \eqref{MFLE} and $\tilde{\I}$ denotes the energy dissipation functional.
Because the diffusion is partially degenerate, we work with the Otto-Onsager metric $W_J$ associated with a symmetric, positive semidefinite matrix $J$, defined by
\begin{equation*}
\begin{aligned}
W_{J}^2(\mu_0, \mu_1):=\inf_{(\mu_t, w_t)} \Big\{ & \int_0^1 \int_{\R^{2d}} \langle w_t, J w_t \rangle \rmd \mu_t \rmd t \bigg|  \partial_t \mu_t+{\rm div} \big(J w_t \mu_t \big) =0 \\
& \text{ holds in the distributional sence} \Big\}
\end{aligned}
\end{equation*}
on the probability measure space $\mathcal{P}(\R^{2d})$.
Using the associated tangent space construction in Definition \ref{TS}, we introduce the corresponding metric derivative and obtain
\begin{equation*}
|\dot{\mu}_{t_0}|=\lim_{t \downarrow t_0} \dfrac{W_J(\mu_t, \mu_{t_0})}{t-t_0} =(\tilde{\I}(u_{t_0}))^{\frac{1}{2}}.
\end{equation*}
In fact, this quantity is not well defined in the original state space, which further underscores the necessity of the initial pullback.
Finally, by introducing suitable perturbations and comparing the Otto-Onsager metric slopes of the resulting curves, we derive a generalized Wasserstein gradient flow formulation for the pulled back system, which can be written as
\begin{equation*}
\begin{aligned}
\partial_t u_t= & -{\rm grad}_{W_J} \tilde{\F}_t(u_t)
=  {\rm div}\Big( J u_t \nabla \big( \ln u_t+ \beta m h_{f_t}(\Phi_t(\tilde{x}, \tilde{v})) \big) \Big),
\end{aligned}
\end{equation*}
where $h_f(x,v):=\frac{1}{2}v^2 +\frac{1}{m}U(x)+\frac{1}{m} K*f(x)$.

Transferring the preceding results to the original system, we recover the GENERIC structure
\begin{equation*}
\begin{aligned}
\partial_t f_t = & \L(f_t) \dfrac{\delta H_{f}}{\delta f}(f_t) + \J (f_t) \dfrac{\delta \F}{\delta f}(f_t) \\
 =& -\{ f_t, h_{f_t} \}+ {\rm div}\Big (\frac{1}{2}GG^{\top} f_t \nabla(\ln f_t +\beta m h_{f_t}) \Big).
\end{aligned}
\end{equation*}
The functional $H_f$ is defined in \eqref{Herergy}.
The operators $\L$ and $\J$ are defined in \eqref{L} and \eqref{O}, respectively.
Accordingly, we obtain the trajectorial rate of free energy dissipation identity, given by
\begin{equation*}
\begin{aligned}
& \lim_{t \downarrow t_0} \dfrac{\mathbb{E}[\theta_{t}(X_{t}, V_t, \rho_{t})| \mathcal{G}_{t_0}] -\theta_{t_0}(X_{t_0}, V_{t_0}, \rho_{t_0}) }{t-t_0} \\
= &2\dfrac{\gamma}{m} +\dfrac{\sigma^2}{2 f_{t_0}(X_{t_0}, V_{t_0})} \Delta_v f_{t_0}(X_{t_0}, V_{t_0}) -\beta\gamma |V_{t_0}|^2 +\dfrac{1}{2}\sigma^2 \Delta_v \ln f_{t_0}(X_{t_0}, V_{t_0}) \\
 & -\dfrac{1}{2}\beta \big \langle \nabla K*f_{t_0}(X_{t_0}), V_{t_0} \big \rangle
 + \E_{\bar{\P}} \Big[ \Big\langle \dfrac{1}{2}\beta \nabla K(X_{t_0}-Y_{t_0}^1), V_{t_0} \Big \rangle  \Big], ~ \P-a.s.,
\end{aligned}
\end{equation*}
as well as the free energy dissipation
\begin{equation}\label{diss2}
\dfrac{\rmd}{\rmd t}\F(f_t)= -\dfrac{1}{2} \int_{\R^{2d}} | \nabla (\ln f_t+ \beta m \check{h}_{f_t} ) G |^2 f_t \rmd x \rmd v =:-\I(f_t)> -\infty,
\end{equation}
where $\theta_t(x,v, f_t):=\ln f_t(x,v) + \beta m f_t \check{h}_{f_t}(x,v)$ and $(Y^1, Y^2)^{\top}$ denotes an independent copy of $(X, V)^{\top}$ defined on $(\bar{\Omega}, \bar{\mathcal{G}}, \bar{\P})$.
We observe that the trajectory based approach provides a more transparent description of how nonlinear interactions shape the dynamics.
The contribution of the linear term $U$ is not explicitly detectable either at the level of individual trajectories or in the total rate of energy dissipation.
By contrast, the influence of the nonlinear interaction term $K$ is manifested directly through the trajectorial dissipation of energy.
This discrepancy stems from the fact that the nonlinear term renders the conservative energy macroscopically conserved, whereas such conservation typically fails along individual trajectories.

Furthermore, partial degeneracy can induce a distinctive geometric structure and lead to a new result.
Upon imposing $\nabla K=0$, we observe that partial degeneracy induces a decomposition of the metric space $(\mathcal{P}(\R^{2d}), W_J)$.
On each piece, a modified Wasserstein distance is well defined, while the distance between distinct pieces is infinite.
Leveraging this geometric decomposition, we further derive a partial HWI inequality of the form
\begin{equation*}
H(\mu_0 | \mu_{\infty})- H(\mu_1 | \mu_{\infty}) \leq \sqrt{I(\mu_0 | \mu_{\infty})}~ W_J(\mu_0, \mu_1)- \dfrac{\kappa_M}{2}W_J^2(\mu_0, \mu_1)
\end{equation*}
under the Bakry-\'Emery criterion.
Here $H$ denotes the relative entropy and $I$ denotes the partially relative Fisher information. See \eqref{rentrop} and \eqref{fisher} for their definitions.
The explicit Bakry-\'Emery condition is stated in Assumption \ref{hess}.

The paper is organized as follows. Section \ref{pre} contains the preliminaries, including notation and assumptions.
Section \ref{main} presents our main results.
In Subsection \ref{re}, we analyze the reversible component. We construct the associated Hamiltonian flow and derive the evolution equation obtained after the pullback.
In Subsection \ref{GF}, we establish the free energy dissipation identity and define the associated metric derivative, which together yield the generalized Wasserstein gradient flow formulation.
Building on this gradient flow representation of the pulled back system, Subsection \ref{generic} recovers the GENERIC structure of the original system \eqref{VFPE}.
Section \ref{degen} establishes a partial HWI inequality based on the metric space structure induced by partial degeneracy.


\section{Preliminaries}\label{pre}
In this section we introduce the notation, functional framework, and standing assumptions used throughout the paper.

\subsection{Notation}
Here we collect the key symbols used throughout the paper.
Most of them are standard, but we include them for the reader's convenience.
Whenever a symbol requires a more precise definition, this will be given at its first occurrence in the text.

\begin{longtable}{@{}p{0.22\textwidth} p{0.55\textwidth} p{0.18\textwidth}@{}}
\toprule
\textbf{Symbol} & \textbf{Meaning} & \textbf{Reference} \\
\midrule
\endfirsthead

\toprule
\textbf{Symbol} & \textbf{Meaning} & \textbf{Reference} \\
\midrule
\endhead

$b_H, b_D, G, A$; $\tilde{b}, \tilde{G}, \tilde{A}$ & drift (reversible and dissipative) coefficients, diffusion coefficients and diffusion matrices & \eqref{bHbDG}, \eqref{coeff} \\
$\rho, f$ & measure and its density for the original system & \\
$\mu, u$ & measure and its density for the pulled back system & \\
$\rho_{\infty}, f_{\infty}, \mu_{\infty}, u_{\infty}$ & invariant measures and their densities & \eqref{infty} \\
$\Phi$ & flow generated by $b_H$ & \eqref{flow} \\
$\mathcal{P}$ & space of probability measures & \\
$\mathcal{P}_2$ & space of probability measures with finite second moment & \\
$H_f$ & energy functional corresponding to $b_H$ & \eqref{Herergy} \\
$\F$, $\tilde{\F}$ & free energy functionals & \eqref{fenergy}, \eqref{penergy} \\
$\I$, $\tilde{\I}$ & energy dissipation functionals & \eqref{diss2}, \eqref{diss} \\
$\theta$, $\tilde{\theta}$ & energy processes &  Thm. \ref{ori}, \eqref{energyp} \\
$\mathcal{D}, \tilde{\mathcal{D}}$ & energy dissipation terms of trajectories & \eqref{cortra}, \eqref{D} \\
$\L$ & Poisson operator &  \eqref{L} \\
$\J$ & Onsager operator & \eqref{O} \\
$\langle \cdot, \cdot \rangle_{J, \mu}$, $\| \cdot \|_{J, \mu}$ & Otto-Onsager scalar product and norm & \eqref{product}, \eqref{norm} \\
$W_{J} $ & Otto-Onsager metric & Def. \ref{metric} \\
${\rm Tan}_{J, \mu} \mathcal{P}(\R^{2d}), |\cdot |_{J, \mu}$ & tangent space of $(\mathcal{P}(\R^{2d}), W_J)$ and the metric on it & Def. \ref{TS} \\
${\rm grad}_{W_J} \tilde{\F}$ & $W_J$-gradient flow of $\tilde{\F}$ & \eqref{grad} \\
$H(\mu | \mu_{\infty})$, $I(\mu | \mu_{\infty})$ & relative entropy, partially relative Fisher information of $\mu$ with respect to $\mu_{\infty}$ & \eqref{rentrop}, \eqref{fisher} \\
\bottomrule
\end{longtable}
To treat the reversible and irreversible components separately, we introduce the reversible operator $L_{rev}$ and irreversible operator $L_{irr}$, defined by
\begin{equation*}
\begin{aligned}
L_{rev}(f):=& \dfrac{1}{m}(\nabla_x U+\nabla_x K*f) \nabla_v f -v\nabla_x f, \\
L_{irr}(f):=& \dfrac{\gamma}{m}\nabla_v  (vf) +\dfrac{1}{2}\sigma^2 \Delta_v f.
\end{aligned}
\end{equation*}
To distinguish symbols that play analogous roles in different systems, we consistently denote quantities in the pulled back system by an overtilde $\tilde{\cdot}$.
For clarity, we record two pieces of notation that are easily conflated. Specifically,
\begin{equation*}
h_f(z)=\dfrac{1}{2}v^2+\dfrac{1}{m}U(x) +\dfrac{1}{m}K*f(x), ~\check{h}_f(z)=\dfrac{1}{2}v^2+\dfrac{1}{m}U(x) +\dfrac{1}{2m}K*f(x).
\end{equation*}
For notational convenience, we write $\E$ for expectation with respect to the probability measure $\P$. Expectations taken under other probability measures are denoted explicitly.
We also adopt the Einstein summation convention.

Next, we briefly recall the definitions of the Poisson bracket and the Poisson-Vlasov bracket, which can be found in Section 1 of \cite{MR}.
For functions $f_1,f_2$ on $\R^d\times \R^d$, the Poisson bracket is defined by
\begin{equation}\label{poi}
\{f_1, f_2\}:= \nabla_x f_1 \nabla_v f_2-\nabla_v f_1 \nabla_x f_2.
\end{equation}
For the Vlasov case, the Poisson bracket between two functionals
$F_1$ and $F_2$ of $f$ is defined by
\begin{equation}\label{poiv}
\{F_1, F_2\}(f):= \int_{\R^{d}\times \R^d} f \Big\{ \dfrac{\delta F_1}{\delta f}, \dfrac{\delta F_2}{\delta f} \Big\}_{x,v}\rmd x \rmd v
\end{equation}
where $f \in C^{\infty}(\R^{2d})$ is the density function and $\{ \cdot, \cdot \}_{x, v}$ denotes the canonical Poisson bracket with respect to $(x,v)$.

\subsection{Setting}

Fix a terminal time $T\in (0, \infty)$. Introduce the notation $z:=(x,v)^{\top}$ and $Z_t:=(X_t,  V_t)^{\top}$.
With this convention, \eqref{VFPE} and \eqref{MFLE} can be written in the form
\begin{equation*}\label{VFP2}
\partial_t f_t(z)=-\nabla \big((b_H(z)+b_D(z))f_t(z)\big) +\dfrac{1}{2}\partial_{ij} (A^{ij}f_t(z))
\end{equation*}
and
\begin{equation*}\label{MFLE2}
\rmd Z_t=(b_H+b_D)(Z_t)\rmd t +G \rmd B_t
\end{equation*}
respectively, where
\begin{equation}\label{bHbDG}
b_H(z):=\binom{v}{-\frac{1}{m}\nabla U(x) -\frac{1}{m} \nabla K*f(x)},~ b_D(z):=\binom{0}{-\frac{\gamma}{m}v},~ G:=\binom{0}{\sigma {\rm I}_{d\times d}},~ A:=GG^{\top}.
\end{equation}

To guarantee sufficient regularity, we impose the following assumptions:

\begin{ass}\label{reg} \quad \\
(i) The smooth potential functions $U, K: \R^d \rightarrow [0, \infty)$ are assumed to satisfy that $K$ is symmetric and that there exists a constant $\kappa_U>0$ such that
\begin{equation*}\label{convex}
U(x)\geq \kappa_U|x|^2 ~ \text{ for all } x\in \R^d.
\end{equation*}
Moreover, $\nabla U$ and $\nabla K$ are Lipschitz continuous, and there exist constants $C_j$ such that
\begin{equation*}\label{der}
|\nabla^j U|\leq C_j, ~ |\nabla^j K|\leq C_j, ~ \forall j \geq 2.
\end{equation*}
(ii) The initial condition satisfies $f_0\in C^{\infty}(\R^{2d})$,
$$\E [|\nabla_v \ln (f_0(X_0, V_0)) |^2]<+\infty, ~
\E[|X_0|^{4}+|V_0|^{4}]< +\infty,
$$ and the initial free energy is finite, i.e. $\mathcal{F}(f_0)<+\infty$.
\end{ass}

Under Assumption \ref{reg}, the mean-field Langevin equation \eqref{MFLE} admits a unique, strong solution $(X_t, V_t)_{0\leq t\leq T}$, see Theorem 1.7 in \cite{C}.
In addition, $\rho_t\in \mathcal{P}_2(\R^{2d})$ for all $t\in [0,T]$,
and the solution
satisfies
\begin{equation}\label{moment}
\E\big[\sup_{0\leq t\leq T}(|X_t|^{p} +|V_t|^{p})\big]<+\infty, ~p=2,4.
\end{equation}
The corresponding density $f_t$ is a smooth solution of the Vlasov-Fokker-Planck equation \eqref{VFPE} (see, for instance, \cite{KRTY}, Lemma 4.3).
Moreover, Theorem 2.9 and Lemma 4.4 in \cite{KRTY} imply that,
for all $t\in (0,T]$, the density $f_t$ is strictly positive and satisfies
\begin{equation}\label{intln}
\E\big[ \int_t^T|\nabla_v \ln (f_s(X_s, V_s)) |^2 \rmd s \big] < +\infty.
\end{equation}

Furthermore, \eqref{MFLE} admits a unique invariant measure $\rho_{\infty}$. Its density $f_{\infty}$ admits an implicit representation of
\begin{equation}\label{infty}
f_{\infty}=\dfrac{\rme^{-\beta m h_{f_{\infty}}}}{\int_{\R^{2d}} \rme^{-\beta m h_{f_{\infty}}} \rmd z},
\end{equation}
and it is also the unique stationary solution of \eqref{VFPE}. Without loss of generality, we assume that
$\int_{\R^{2d}} \rme^{-\beta m h_{f_{\infty}}} \rmd z=1$.
In what follows, we adopt the convention $\mathcal{F}(\rho):=\mathcal{F}(f)$, where $\rmd \rho=f \rmd z$.
We use the same convention for $\I$, $\tilde{\I}$ and $\tilde{\F}$.

\section{Main results}\label{main}
In this section, we present the main results. We begin by analyzing the reversible component and constructing the associated flow used to pull back the dynamics.
We then establish a generalized Wasserstein gradient flow formulation for the pulled back system.
Finally, we transfer these results to the original system \eqref{VFPE} and recover its GENERIC structure.

\subsection{Hamiltonian flow} \label{re}
Before analyzing the energy dissipation and constructing a gradient flow formulation, we first isolate the reversible component of \eqref{VFPE}.
Specifically, we pull back the reversible dynamics along each trajectory, which yields a clean separation between the reversible and dissipative contributions to the evolution.

Let $\Phi$ denote the flow on $\R^{2d}$ generated by $b_H$, defined by
\begin{equation}\label{flow}
\begin{aligned}
\dfrac{\rmd}{\rmd t}\Phi_t(z)&=b_H(\Phi_t(z)), \\
\Phi_0(z)&=z.
\end{aligned}
\end{equation}
The associated energy functional is given by
\begin{equation}\label{Herergy}
\begin{aligned}
H_f(g):=&\int_{\R^{2d}} g \check{h}_f \rmd x\rmd v \\
=&\int_{\R^{2d}} g(x,v) \Big( \dfrac{1}{2}v^2 +\dfrac{1}{m}U(x)+\dfrac{1}{2m} K*f(x)\Big) \rmd x\rmd v.
\end{aligned}
\end{equation}

\begin{prop}
The flow $\Phi$ is the Hamiltonian flow generated by $H_f$.
\end{prop}

\begin{proof}
A direct calculation shows that
\begin{equation*}\label{henergy}
\dfrac{\delta H_f}{\delta f}(f)=\dfrac{1}{2}v^2+\dfrac{1}{m}U(x) +\dfrac{1}{m}K*f(x)=h_f
\end{equation*}
and
\begin{equation*}
L_{rev}(f)=\dfrac{1}{m}(\nabla_x U+\nabla_x K*f) \nabla_v f -v\nabla_x f=-\{f, h_f \}
\end{equation*}
by \eqref{poi}.

Let $\phi\in C_c^{\infty}(\R^{2d})$ and define
\begin{equation*}
\Lambda_{\phi}(f):=\int_{\R^{2d}} \phi f \rmd z.
\end{equation*}
By the definition of the Poisson-Vlasov bracket in \eqref{poiv}, we have
\begin{equation*}
\begin{aligned}
\{ \Lambda_{\phi}, H_f \}(f)=&\int_{\R^{2d}} f \Big\{ \dfrac{\delta \Lambda_{\phi}}{\delta f}(f), \dfrac{\delta H_f}{\delta f}(f) \Big\}  \rmd z \\
=& \int_{\R^{2d}} f \{ \phi, h_f \}\rmd z \\
=& \int_{\R^{2d}} f (\nabla_x \phi \cdot \nabla_v h_f-\nabla_v \phi \cdot \nabla_x h_f)\rmd z.
\end{aligned}
\end{equation*}
Integration by parts in $x$ and $v$ yields
\begin{equation*}
\begin{aligned}
\{ \Lambda_{\phi}, H_f \}(f)=\int_{\R^{2d}} -\phi \nabla_x f \cdot \nabla_v h_f+\phi \nabla_v f \cdot \nabla_x h_f \rmd z
=\int_{\R^{2d}} -\phi \{ f, h_f \}\rmd z.
\end{aligned}
\end{equation*}
Therefore, for every test function $\phi$, we obtain
\begin{equation*}
\dfrac{\rmd}{\rmd t}\Lambda_{\phi}(f)=\int_{\R^{2d}} \phi  L_{rev}(f) \rmd z
=\int_{\R^{2d}} -\phi \{ f, h_f \}\rmd z
=\{ \Lambda_{\phi}, H_f \}(f),
\end{equation*}
which completes the proof.
\end{proof}

\begin{rem}\label{hrem}
(i) For $\dot{z}=b_H(z)=Q\cdot \nabla h_f(z)$, where $Q=\left(\begin{smallmatrix} 0 & {\rm I}_{d\times d} \\ -{\rm I}_{d\times d} & 0 \end{smallmatrix}\right)$, we compute that
\begin{equation}\label{energyc}
\dfrac{\rmd}{\rmd t}h_f=\partial_t h_f +\nabla_x h_f\cdot \dot{x} +\nabla_v h_f\cdot \dot{v}=\partial_t h_f
\end{equation}
and
\begin{equation}\label{energyc2}
\dfrac{\rmd}{\rmd t}H_f=\int_{\R^{2d}} \dfrac{\delta H_f}{\delta f}(f) L_{rev}(f) \rmd z=- \int_{\R^{2d}} h_f \{f, h_f \}\rmd z=- \dfrac{1}{2}\int_{\R^{2d}} \{f, h_f^2 \}\rmd z =0.
\end{equation}
The identity $\partial_t h_f=0$ holds only in the special case $\nabla K=0$.
In general, the distribution-dependent term $\nabla K*f$ obstructs a proof of trajectory-wise conservation of the particle energy.
Instead, one can establish energy conservation only at the macroscopic level. \\
(ii) Since $\nabla \cdot b_H(z) =\nabla_x \cdot (\nabla_v h_f) +\nabla_v \cdot (-\nabla_x h_f)=0$, it follows that
\begin{equation*}
\dfrac{\rmd}{\rmd t}|D\Phi_t| =|D\Phi_t| \tr (D b_H(\Phi_t))= |D\Phi_t| \nabla\cdot b_H(\Phi_t)=0, ~ t \in [0, T].
\end{equation*}
Together with $|D\Phi_0|=1$, this implies that $|D\Phi_t|=1$ for all $ t \in [0, T]$.
\end{rem}

We give the Poisson operator $\L$ based on
\begin{equation*}
\{F_1, F_2\}(f)= \int_{\R^{d}\times \R^d}  \dfrac{\delta F_1}{\delta f} \L(f) \dfrac{\delta F_2}{\delta f} \rmd x \rmd v.
\end{equation*}
Here $F_1$ and $F_2$ are functionals of $f$.
Using
\begin{equation*}
\begin{aligned}
\{F_1, F_2\}(f)= &\int_{\R^{d}\times \R^d} f \Big\{ \dfrac{\delta F_1}{\delta f}, \dfrac{\delta F_2}{\delta f} \Big\} \rmd x \rmd v \\
=&\int_{\R^{d}\times \R^d} \dfrac{\delta F_1}{\delta f} \Big(-\Big\{ f , \dfrac{\delta F_2}{\delta f} \Big\}\Big) \rmd x \rmd v,
\end{aligned}
\end{equation*}
we can define the Poisson operator $\L$ at $f$ by
\begin{equation}\label{L}
\L(f)\phi:=-\{f, \phi\},
\end{equation}
where $\phi \in C^{\infty}(\R^{2d})$.
Then
\begin{equation*}
L_{rev}(f_t)=-\{ f_t, h_{f_t} \} =\L(f_t)\dfrac{\delta H_{f}}{\delta f}(f_t).
\end{equation*}

We now pull back the systems \eqref{MFLE} and \eqref{VFPE} along the flow $\Phi$.
Define $\tilde{Z}_t:=(\tilde{X}_t, \tilde{V}_t)^{\top}:=\Phi_{-t}(Z_t)$ and $\tilde{z}:=(\tilde{x}, \tilde{v})^{\top}:=\Phi_{-t}(z)$.
A direct computation yields
\begin{equation*}
\begin{aligned}
\rmd \tilde{Z}_t^{\alpha}= &\big [-\partial_i\Phi_{-t}^{\alpha}(Z_t)b_H^i(Z_t)\rmd t +\partial_i\Phi_{-t}^{\alpha}(Z_t)\rmd Z_t^i +\dfrac{1}{2}\partial_{ij}\Phi_{-t}^{\alpha}(Z_t)A^{ij}\rmd t\big ] \Big |_{Z_t=\Phi_t(\tilde{Z}_t)} \\
= & \big [(D\Phi_{-t} b_D)^{\alpha}(Z_t) \rmd t +(D\Phi_{-t}(Z_t) G\rmd B_t)^{\alpha} +\dfrac{1}{2}\tr(A D^2\Phi_{-t}^{\alpha}(Z_t)) \rmd t \big ] \Big |_{Z_t=\Phi_t(\tilde{Z}_t)},
\end{aligned}
\end{equation*}
where $\alpha=1, ... , 2d$.
We denote $\Gamma^{\alpha}:=\tr(A D^2\Phi_{-t}^{\alpha}(Z_t))$ and $\Gamma:=(\Gamma^1, ... , \Gamma^{2d})^{\top}$, and thus
\begin{equation}\label{MFLE3}
\rmd \tilde{Z}_t=\tilde{b}(\tilde{Z}_t)\rmd t+\tilde{G}(\tilde{Z}_t)\rmd B_t,
\end{equation}
where
\begin{equation}\label{coeff}
\tilde{b}(\tilde{Z}_t):=D\Phi_{-t}(\Phi_t(\tilde{Z}_t)) b_D(\Phi_t(\tilde{Z}_t)) +\dfrac{1}{2}\Gamma(\Phi_t(\tilde{Z}_t)),~ \tilde{G}(\tilde{Z}_t):=D\Phi_{-t}(\Phi_t(\tilde{Z}_t)) G.
\end{equation}
For convenience, we write $\tilde{A}:=\tilde{G}\tilde{G}^{\top}$, and we use $\mu_t$ to denote the law of \eqref{MFLE3}.

Similarly, we obtain
\begin{equation*}
\begin{aligned}
\partial_t f_t(\Phi_t(\tilde{z}))
= & \big[ \partial_t f_t(z) +\nabla f_t(z)\cdot b_H(z)\big] \Big|_{z=\Phi_t(\tilde{z})} \\
= & \big [ -\nabla \cdot (b_D(z))f_t(z)) +\dfrac{1}{2}\partial_{ij}(A^{ij}f_t(z)) \big] \Big|_{z=\Phi_t(\tilde{z})} \\
= & -\partial_i(\tilde{b}^i(\tilde{z}) f_t(\Phi_t(\tilde{z})) ) +\dfrac{1}{2}\partial_{ij}((\tilde{G}\tilde{G}^{\top})^{ij}(\tilde{z}) f_t(\Phi_t(\tilde{z})) ).
\end{aligned}
\end{equation*}
Set $u_t(\tilde{z}):= f_t(\Phi_t(\tilde{z}))$.
The system obtained from \eqref{VFPE} can then be written as
\begin{equation}\label{VFPE2}
\begin{aligned}
\partial_t u_t(\tilde{z})
=  -\partial_i(\tilde{b}^i(\tilde{z}) u_t(\tilde{z}) ) +\dfrac{1}{2}\partial_{ij}((\tilde{G}\tilde{G}^{\top})^{ij}(\tilde{z}) u_t(\tilde{z})).
\end{aligned}
\end{equation}

\subsection{Gradient flow}\label{GF}

We now turn to the irreversible component and present a gradient flow formulation of \eqref{VFPE2}.
We proceed in three steps. We first establish the free energy dissipation identity.
We then characterize the metric derivative.
Finally, we analyze the metric slope of the free energy.

{\bf Step 1: Free energy dissipation.}
Equation \eqref{energyc2} shows that the total free energy is conserved under the pullback, meaning that $H_f$ is invariant along the flow $\Phi$.
Equation \eqref{energyc} indicates, however, that this conservation generally fails at the level of individual trajectories.
Consequently, the definition of the energy process along a given trajectory must explicitly incorporate the pullback.
We denote the functional
\begin{equation*}\label{pbenergy}
\tilde{\F}_t(u)=\int_{\R^{2d}} u(\tilde{z}) \ln u(\tilde{z})+ \beta m u(\tilde{z}) \check{h}_{f_t}(\Phi_t(\tilde{z}))  \rmd \tilde{z}
\end{equation*}
and the energy process
\begin{equation*}
\tilde{\theta}_t(\tilde{Z_t}, \rho_t):=\ln u_t(\tilde{Z_t}) + \beta m \check{h}_{\rho_t}(\Phi_t(\tilde{Z_t})).
\end{equation*}
In view of \eqref{fenergy} and the flow $\Phi$, we obtain
\begin{equation*}
\F(f_t) =\tilde{\F}_t(u_t)
=\E[\tilde{\theta}_t(\tilde{Z}_t, \rho_t)].
\end{equation*}
Here, our focus is on the rate of dissipation of $\tilde{\F}_t(u_t)$, rather than of $\F(u_t)$, since our primary interest lies in the rate of energy dissipation in \eqref{VFPE}.

\begin{thm}\label{fder}
Suppose that Assumption \ref{reg} holds. Then, for $t_0\in[0,T)$, the trajectorial rate of free energy dissipation identity is given by
\begin{equation}\label{thm}
\begin{aligned}
\lim_{t \downarrow t_0} \dfrac{\mathbb{E}[\tilde{\theta}_{t}(\tilde{Z}_{t}, \rho_{t})| \mathcal{G}_{t_0}] -\tilde{\theta}_{t_0}(\tilde{Z}_{t_0}, \rho_{t_0}) }{t-t_0}
= \tilde{\mathcal{D}}_{t_0}(\tilde{Z}_{t_0}), ~ \P-a.s.,
\end{aligned}
\end{equation}
where
\begin{equation}\label{D}
\begin{aligned}
\tilde{\mathcal{D}}_t(\tilde{z}):=& \dfrac{1}{2} \nabla \tilde{\theta}_t(\tilde{z}, f_t) \Big( {\rm div} \tilde{A}(\tilde{z}) +\tilde{A}(\tilde{z}) \nabla \ln u_t(\tilde{z}) +\beta m\tilde{A}(\tilde{z}) \nabla \check{h}_{f_t}(\Phi_t(\tilde{z}) ) +\Gamma(\Phi_t(\tilde{z}) ) \Big) \\
&+\tr (\tilde{A}(\tilde{z}) \Delta \tilde{\theta}_t(\tilde{z}, f_t))
+ \beta m \nabla \check{h}_{f_t}(\Phi_t(\tilde{z})) b_H( \Phi_t(\tilde{z})) \\
& + \dfrac{\beta}{2}
\int_{\R^{2d}} \binom{\nabla K(\Phi_t^1(\tilde{z})-y^1)}{0} b_H(y) f_t(y) \rmd y
\end{aligned}
\end{equation}
satisfies
\begin{equation}\label{thm1}
\begin{aligned}
&\E[\int_0^T \! \tilde{\mathcal{D}}_t(\tilde{Z}_{t})  \rmd t] \\
=& -\dfrac{1}{2} \int_0^T \E \big[ \big|\nabla \big(\ln u_t(\tilde{Z}_t) +\beta m \check{h}_{\rho_t}(\Phi_t(\tilde{Z}_t)) \big) D\Phi_{-t}(\Phi_t(\tilde{Z_t}))G \big|^2 \big ] \rmd t>-\infty.
\end{aligned}
\end{equation}
Here $\Phi_t(\tilde{z})=(\Phi_t^1(\tilde{z}), \Phi_t^2(\tilde{z}))^{\top}$, $y=(y^1, y^2)^{\top} \in \R^{2d}$ and
$Y:=(Y^1, Y^2)^{\top}$ denotes an independent copy of $Z$ defined on $(\bar{\Omega}, \bar{\mathcal{G}}, \bar{\P})$.
Moreover, the free energy dissipation is
\begin{equation}\label{average}
\begin{aligned}
\dfrac{\rmd}{\rmd t}\tilde{\F}_t(u_t)= & -\dfrac{1}{2} \int_{\R^{2d}} \big| \nabla \big(\ln u_t(\tilde{z}) +\beta m \check{h}_{f_t}(\Phi_t(\tilde{z})) \big) D\Phi_{-t}(\Phi_t(\tilde{z}))G \big|^2 u_t(\tilde{z}) \rmd \tilde{z} \\
=& -\tilde{\I}(u_t).
\end{aligned}
\end{equation}
\end{thm}

\begin{proof}
Since we will apply It\^{o}'s formula for McKean-Vlasov stochastic differential equations (see, for instance, \cite{BLPR}, Theorem 7.1 or \cite{CD}, Proposition 5.102), we first introduce an independent copy of $Z$.
Specifically, let $Y=(Y^1, Y^2)^{\top}$ be an independent copy of $Z$ defined on $(\bar{\Omega}, \bar{\mathcal{G}}, \bar{\P})$.

Recalling that $\tilde{\theta}_t(\tilde{z}, f_t)=\ln u_t(\tilde{z}) + \beta m \check{h}_{f_t}(\Phi_t(\tilde{z})) $, we obtain
\begin{equation*}
\begin{aligned}
\rmd \tilde{\theta}_t(\tilde{Z}_t, \rho_t)= &\partial_t \tilde{\theta}_t(\tilde{Z}_t, \rho_t) \rmd t +\partial_{\tilde{z}}\tilde{\theta}_t(\tilde{Z}_t, \rho_t) \rmd \tilde{Z}_t +\dfrac{1}{2}\partial_{ij}\tilde{\theta}_t(\tilde{Z}_t, \rho_t) \tilde{A}^{ij}(\tilde{Z}_t) \rmd t \\
&+\E_{\bar{\P}}\Big[ \big\langle \partial_{\rho}\tilde{\theta}_t(\tilde{Z}_t, \rho_t), (b_H+b_D) \big \rangle (Y_t)+\dfrac{1}{2}\tr \big(\partial_y\partial_{\rho}\tilde{\theta}_t(\tilde{Z}_t, \rho_t)(Y_t) A \big) \Big] \rmd t.
\end{aligned}
\end{equation*}
We denote
\begin{equation}\label{decomposition}
\rmd \tilde{\theta}_t(\tilde{Z}_t, \rho_t)=({\rm I}+{\rm III})\rmd t + ({\rm II}) \rmd B_t,
\end{equation}
where
\begin{equation*}
\begin{aligned}
{\rm I}:= &\partial_t \tilde{\theta}_t(\tilde{Z}_t, \rho_t)
+\nabla \tilde{\theta}_t(\tilde{Z}_t, \rho_t) \tilde{b}(\tilde{Z}_t) +\dfrac{1}{2}\partial_{ij}\tilde{\theta}_t(\tilde{Z}_t, \rho_t) \tilde{A}^{ij}(\tilde{Z}_t), \\
{\rm II}:= &\nabla \tilde{\theta}_t(\tilde{Z}_t, \rho_t) \tilde{G}(\tilde{Z}_t), \\
{\rm III}:= &\E_{\bar{\P}}\Big[ \big\langle \partial_{\rho}\tilde{\theta}_t(\tilde{Z}_t, \rho_t), (b_H+b_D) \big \rangle (Y_t)  +\dfrac{1}{2}\tr \big(\partial_y\partial_{\rho}\tilde{\theta}_t(\tilde{Z}_t, \rho_t) (Y_t) A \big) \Big].
\end{aligned}
\end{equation*}

Since $G$ is diagonal and $\tilde{A}(\tilde{z})=D\Phi_{-t}(\Phi_t(\tilde{z}))G G^{\top} (D\Phi_{-t}(\Phi_t(\tilde{z})))^{\top}$, it follows that
\begin{equation*}
\partial_{ij}\tilde{\theta}_t(\tilde{Z}_t, \rho_t) \tilde{A}^{ij}(\tilde{Z}_t)=\tr(\tilde{A}(\tilde{Z}_t)\Delta \tilde{\theta}_t (\tilde{Z}_t, \rho_t) ).
\end{equation*}
A direct calculation yields
\begin{equation*}
\begin{aligned}
b_D(\Phi_t(\tilde{Z}_t))= & -\dfrac{\beta m}{2}D\Phi_{-t}(\Phi_t(\tilde{Z}_t))^{\top} GG^{\top} \nabla h_{f_t}(\Phi_t(\tilde{Z}_t)) \\
=& -\dfrac{\beta m}{2}D\Phi_{-t}(\Phi_t(\tilde{Z}_t))^{\top} GG^{\top} \nabla \check{h}_{f_t}(\Phi_t(\tilde{Z}_t)).
\end{aligned}
\end{equation*}
Moreover, by the definition of the flow $\Phi$, we have
\begin{equation*}
\dfrac{\rmd}{\rmd t} \Phi_t(\tilde{Z}_t) =b_H(\Phi_t(\tilde{Z}_t)).
\end{equation*}
Therefore, we obtain
\begin{equation}\label{i}
\begin{aligned}
{\rm I}=& \dfrac{1}{2u_t(\tilde{Z}_t)}{\rm div}\big(\tilde{A}(\tilde{Z}_t) u_t(\tilde{Z}_t) \nabla \tilde{\theta}_t(\tilde{Z}_t, \rho_t)\big)
+ \beta m \nabla \check{h}_{\rho_t}(\Phi_t(\tilde{Z_t})) \dfrac{\rmd }{\rmd t} \Phi_t(\tilde{Z_t}) \\
& +\nabla \tilde{\theta}_t(\tilde{Z}_t, \rho_t) \tilde{b}(\tilde{Z}_t) +\dfrac{1}{2}\tr(\tilde{A}(\tilde{Z}_t) \Delta \tilde{\theta}_t(\tilde{Z}_t, \rho_t)) \\
=& \dfrac{1}{2} \nabla \tilde{\theta}_t(\tilde{Z}_t, \rho_t) \Big( {\rm div} \tilde{A}(\tilde{Z}_t) +\tilde{A}(\tilde{Z}_t) \nabla \ln u_t(\tilde{Z}_t) +\beta m \tilde{A}(\tilde{Z}_t) \nabla \check{h}_{\rho_t}(\Phi_t(\tilde{Z}_t)) +\Gamma(\Phi_t(\tilde{Z}_t))  \Big)\\
& + \beta m \nabla \check{h}_{\rho_t}(\Phi_t(\tilde{Z_t})) b_H( \Phi_t(\tilde{Z_t})) +\tr (\tilde{A}(\tilde{Z}_t) \Delta \tilde{\theta}_t(\tilde{Z}_t, \rho_t)).
\end{aligned}
\end{equation}
Using the Lions derivative (see Subsection 5.2 of \cite{CD}),
we have
\begin{equation}\label{iii}
\begin{aligned}
{\rm III}=\E_{\bar{\P}} \Big[ \Big\langle \dfrac{\beta}{2}\binom{\nabla K(X_t -Y_t^1)}{0}, (b_H+b_D)(Y_t) \Big \rangle  \Big]
+\E_{\bar{\P}}\Big[ -\dfrac{1}{4} \tr\Big( \partial_y \binom{\nabla K(X_t- Y_t^1)}{0} A \Big)  \Big].
\end{aligned}
\end{equation}
Since the first $d$ components of $b_D$ vanish and the first $d$ columns of the matrix $A$ are zero vectors, it follows that
\begin{equation*}
\begin{aligned}
{\rm III}=\E_{\bar{\P}} \Big[ \Big\langle \dfrac{\beta}{2}\binom{\nabla K(X_t -Y_t^1)}{0}, b_H(Y_t) \Big \rangle  \Big].
\end{aligned}
\end{equation*}
Consequently, we obtain
\begin{equation*}
\rmd \tilde{\theta}_t(\tilde{Z}_t, \rho_t)=\tilde{\mathcal{D}}_t(\tilde{Z}_{t}) \rmd t + \nabla \tilde{\theta}_t(\tilde{Z}_t, \rho_t) \tilde{G}(\tilde{Z}_t) \rmd B_t.
\end{equation*}

First, we note that
\begin{equation*}
\begin{aligned}
& \E\Big[ \E_{\bar{\P}} \Big[ \Big\langle \binom{\nabla K(X_t -Y_t^1)}{0}, b_H(Y_t) \Big \rangle  \Big] \Big ] \\
= & \E \big[ \langle  \nabla K*f_{t}(X_{t}), Y_{t}^1 \rangle \big] \\
= & \int_{\R^{2d}} \langle \nabla K*f_t(x), v \rangle f_t(z) \rmd z \\
= & 2 m\int_{\R^{2d}}  \nabla \check{h}_{f_t} (\Phi_t(\tilde{z})) b_H(\Phi_t(\tilde{z})) u_t(\tilde{z}) \rmd \tilde{z}.
\end{aligned}
\end{equation*}
Then for $\tilde{\mathcal{D}}_t(\tilde{Z}_{t})$, we have
\begin{equation*}
\begin{aligned}
\E[\tilde{\mathcal{D}}_t(\tilde{Z}_{t})]
= &\int_{\R^{2d}} \Big [ \dfrac{1}{2}{\rm div}(\tilde{A}(\tilde{z}) u_t(\tilde{z}) \nabla \tilde{\theta}_t(\tilde{z}, f_t)) +\nabla \tilde{\theta}_t(\tilde{z}, f_t) \tilde{b}(\tilde{z}) u_t(\tilde{z}) \\
& +\beta m \nabla \check{h}_{f_t}(\Phi_t(\tilde{z})) b_H( \Phi_t(\tilde{z}))u_t(\tilde{z})
 +\dfrac{1}{2}\tr(\tilde{A}(\tilde{z}) \Delta \tilde{\theta}_t(\tilde{z}, f_t)) u_t(\tilde{z}) \Big ]\rmd \tilde{z} \\
 & + \dfrac{\beta}{2} \E\Big[ \E_{\bar{\P}} \Big[ \Big\langle \binom{\nabla K(X_t -Y_t^1)}{0}, b_H(Y_t) \Big \rangle  \Big] \Big ] \\
= & \int_{\R^{2d}} \Big[  \nabla \tilde{\theta}_t(\tilde{z}, f_t)) D\Phi_{-t}(\Phi_t(\tilde{z})) b_D(\Phi_t(\tilde{z})) u_t(\tilde{z})
+\dfrac{1}{2} \nabla \tilde{\theta}_t(\tilde{z}, f_t)) \Gamma(\Phi_t(\tilde{z})) u_t(\tilde{z}) \\
& +\dfrac{1}{2}\tr(\tilde{A}(\tilde{z}) \Delta \tilde{\theta}_t(\tilde{z}, f_t)) u_t(\tilde{z}) \Big] \rmd \tilde{z}\\
= & \dfrac{1}{2} \int_{\R^{2d}} \Big[ \big \langle \nabla \tilde{\theta}_t(\tilde{z}, f_t)), -\beta m \tilde{A}(\tilde{z}) \nabla \check{h}_{f_t}(\Phi_t(\tilde{z})) \big \rangle
+ \nabla \tilde{\theta}_t(\tilde{z}, f_t)) \Gamma(\Phi_t(\tilde{z})) u_t(\tilde{z}) \\
& - \big \langle \nabla \tilde{\theta}_t(\tilde{z}, f_t)), \tilde{A}(\tilde{z}) \nabla u_t(\tilde{z}) \big \rangle -\nabla \tilde{\theta}_t(\tilde{z}, f_t)) u_t(\tilde{z}) {\rm div}{\tilde{A}(\tilde{z})} \Big]
 \rmd \tilde{z} \\
= & -\dfrac{1}{2} \int_{\R^{2d}} \big \langle \nabla \tilde{\theta}_t(\tilde{z}, f_t)), \tilde{A}(\tilde{z}) \nabla u_t(\tilde{z}) +\beta m \tilde{A}(\tilde{z}) \nabla  \check{h}_{f_t}(\Phi_t(\tilde{z})) \big \rangle
\rmd \tilde{z} \\
= & -\dfrac{1}{2} \int_{\R^{2d}} \big | \nabla \big( \ln u_t(\tilde{z})+\beta m \check{h}_{f_t}(\Phi_t(\tilde{z})) \big) D\Phi_{-t}(\Phi_t(\tilde{z}))G \big |^2 u_t(\tilde{z}) \rmd \tilde{z}.
\end{aligned}
\end{equation*}
Using \eqref{intln} together with $|D\Phi_t|=1$, we deduce
\begin{equation}\label{int}
\E\big[ \int_t^T \big|\nabla_{\tilde{v}} \ln (u_s(\tilde{Z}_s)) \big|^2 \rmd s \big]< +\infty, ~\text{ for all } t\in (0,T].
\end{equation}
Combining this computation with \eqref{moment} yields
\begin{equation}\label{estimate}
\begin{aligned}
& \int_0^T \int_{\R^{2d}} \big| \nabla \big(\ln u_t(\tilde{z})+\beta m \check{h}_{f_t}(\Phi_t(\tilde{z})) \big) D\Phi_{-t}(\Phi_t(\tilde{z}))G \big|^2 u_t(\tilde{z}) \rmd \tilde{z} \rmd t \\
\leq& 2\E\big[ \int_0^T |\nabla_{\tilde{v}} \ln u_t(\tilde{Z}_t) |^2 |D\Phi_{-t}(\Phi_t(\tilde{Z}_t))G |^2 \rmd t \big] \\
 & + 2 \beta^2 m^2\E\big[ \int_0^T \sup_{0\leq t\leq T}|\tilde{V}_t|^2 |D\Phi_{-t}(\Phi_t(\tilde{Z}_t))G |^2 \rmd t \big]
<+\infty.
\end{aligned}
\end{equation}
Therefore, we conclude
\begin{equation}\label{martingale}
\mathbb{E}\big[\big| \tilde{\mathcal{D}}_t(\tilde{Z}_{t}) \big|\big]< +\infty, ~\E [ \langle {\rm II}, {\rm II} \rangle_T ]< +\infty,
\end{equation}
and \eqref{thm1} holds.

Using \eqref{D} and \eqref{decomposition}, we obtain
\begin{equation}\label{57}
\begin{aligned}
\mathbb{E}[\tilde{\theta}_{t}(\tilde{Z}_{t}, \rho_{t})| \mathcal{G}_{t_0}] -\tilde{\theta}_{t_0}(\tilde{Z}_{t_0}, \rho_{t_0})
=\mathbb{E}\Big[\int_{t_0}^{t} \tilde{\mathcal{D}}_s(\tilde{Z}_{s}) \rmd s \Big| \mathcal{G}_{t_0} \Big].
\end{aligned}
\end{equation}
From Jensen's inequality and \eqref{martingale}, we deduce
\begin{equation*}
\begin{aligned}
 \mathbb{E}\Big[\int_{t_0}^{t} \big|\mathbb{E}(\tilde{\mathcal{D}}_s(\tilde{Z}_{s}) | \mathcal{G}_{t_0}) \big| \rmd s\Big]
\leq \int_{t_0}^{t} \mathbb{E}\Big[\mathbb{E}\big(\big| \tilde{\mathcal{D}}_s(\tilde{Z}_{s}) \big| \Big| \mathcal{G}_{t_0}\big)\Big] \rmd s
= \int_{t_0}^{t} \mathbb{E}\big[\big|\tilde{\mathcal{D}}_s(\tilde{Z}_{s}) \big|\big]\rmd s < +\infty,
\end{aligned}
\end{equation*}
and hence
\begin{equation*}
\int_{t_0}^{t} \big|\mathbb{E}(\tilde{\mathcal{D}}_s(\tilde{Z}_{s}) | \mathcal{G}_t) \big| \rmd s < +\infty, ~ \P-a.s.
\end{equation*}
It follows that
\begin{equation*}
\begin{aligned}
 \lim_{t \downarrow t_0}  \dfrac{\mathbb{E}\big [\int_{t_0}^{t} \tilde{\mathcal{D}}_s(\tilde{Z}_{s}) \rmd s \big| \mathcal{G}_{t_0} \big]}{t-t_0}
=  \lim_{t \downarrow t_0} \dfrac{\int_{t_0}^{t} \mathbb{E}\big [\tilde{\mathcal{D}}_s(\tilde{Z}_{s}) \big| \mathcal{G}_{t_0} \big]\rmd s}{t-t_0}
=  \mathbb{E}\big [\tilde{\mathcal{D}}_{t_0}(\tilde{Z}_{t_0}) \big| \mathcal{G}_{t_0} \big]
= \tilde{\mathcal{D}}_{t_0}(\tilde{Z}_{t_0}), ~ \P-a.s.
\end{aligned}
\end{equation*}
Combining this equation with \eqref{57} yields \eqref{thm}.

Taking expectations in \eqref{decomposition} and using \eqref{martingale}, we obtain
\begin{equation*}
\tilde{\F}_t(u_{t})-\tilde{\F}_{t_0}(u_{t_0})=
-\int_{t_0}^{t} \tilde{\I}(u_s) \rmd s.
\end{equation*}
The Lebesgue differentiation theorem then implies \eqref{average}.
The proof is complete.
\end{proof}

{\bf Step 2: Metric derivative.}
Now we use the Onsager operator to select a Wasserstein-type metric.
For the symmetric, positive semidefinite matrix $J:=\frac{1}{2}\tilde{A}$, we define the Onsager operator $\J$ at $u$ by
\begin{equation}\label{O}
\J(u) \psi :=-{\rm div}(J u \nabla \psi),
\end{equation}
where $\psi \in C^{\infty}(\R^{2d})$.
Equation \eqref{VFPE2} can be rewritten as
\begin{equation}\label{CE}
\partial_t u_t+ {\rm div}(J w_t u_t)=0,
\end{equation}
where $w_t(\tilde{z}):=- \nabla \ln u_t(\tilde{z}) -\beta m \nabla h_{f_t}(\Phi_t(\tilde{z}))$.
We define the Otto-Onsager scalar product at $\mu\in \mathcal{P}(\R^{2d})$ by
\begin{equation}\label{product}
\langle \xi_1, \xi_2 \rangle_{J, \mu}
:=\int_{\R^{2d}}\langle \nabla \psi_1, J \nabla \psi_2 \rangle \rmd \mu,
\end{equation}
and the associated norm by
\begin{equation}\label{norm}
\| \xi_i \|^2_{J,\mu}:=\langle \xi_i, \xi_i \rangle_{J, \mu},
\end{equation}
where $\xi_i=-{\rm div}(J \mu \nabla \psi_i), i=1,2$.
The following provides the definition of the metric.

\begin{de}[Otto-Onsager metric]\label{metric}
Given $\mu_0, \mu_1 \in \mathcal{P}(\R^{2d})$, we define {\em the Otto-Onsager metric} as
\begin{equation*}
\begin{aligned}
W_{J}^2(\mu_0, \mu_1):=\inf_{(\mu_t, w_t)} \Big\{ & \int_0^1 \bigg \| \dfrac{\partial \mu_t }{\partial t} \bigg \|^2_{J, \mu_t} \rmd t \bigg| \partial_t \mu_t+{\rm div} \big(J w_t \mu_t \big) =0 \\
& \text{ holds in the distributional sence} \Big\}.
\end{aligned}
\end{equation*}
\end{de}

Evidently, the Otto-Onsager metric is Wasserstein-type.
To ensure that the notions of absolutely continuous curves and tangent space are well-posed, we introduce an equivalence relation $\sim$.
Specifically, if two vector fields $w$ and $\tilde{w}$ satisfy ${\rm div}(J (w-\tilde{w}))=0$ we write $w \sim \tilde{w}$ and denote the corresponding equivalence class by $[w]$.
We then introduce the notions of absolutely continuous curves and the tangent space, following Definition 1.1.1 in \cite{AGS} and Subsection 4.1 of \cite{OTTO}.

\begin{de}[Absolutely continuous curve]\label{AC}
For a curve $\mu:(T_1, T_2)\rightarrow \mathcal{P}(\R^{2d})$, we say that $\mu$ is {\em an absolutely continuous curve} in $(\mathcal{P}(\R^{2d}), W_J)$ if there exists a function
$m\in L^1(T_1, T_2)$ such that
\begin{equation*}
W_J(\mu_{t_1}, \mu_{t_2}) \leq \int_{t_1}^{t_2}m_t\rmd t, ~ \text{for } T_1<t_1\leq t_2<T_2.
\end{equation*}
\end{de}

\begin{de}[Tangent space]\label{TS}
Let $\mu \in \mathcal{P}(\R^{2d})$, we define
\begin{equation*}
{\rm Tan}_{J, \mu} \mathcal{P}(\R^{2d}):=\overline{\{ [\nabla \zeta] : \zeta \in C_c^{\infty}(\R^{2d}) \}}^{| \cdot|_{J, \mu}},
\end{equation*}
where $|[\nabla \zeta]|_{J, \mu}^2:=\int_{\R^{2d}} \langle \nabla\zeta , J \nabla \zeta \rangle \rmd \mu$.
\end{de}

In analogy with the metric derivative analysis in Subsection 4.3 of \cite{OTTO}, we obtain the following lemma.

\begin{lem}\label{derivative}
For an absolutely continuous curve $\mu:(T_1, T_2)\rightarrow \mathcal{P}(\R^{2d})$, the metric derivative satisfies
\begin{equation*}
|\dot{\mu}_{t_0}|:=\lim_{t \downarrow t_0} \dfrac{W_J(\mu_t, \mu_{t_0})}{t-t_0} \leq | w_{t_0} |_{J, \mu_{t_0}} = \Big ( \int_{\R^{2d}} \langle w_{t_0}, J w_{t_0} \rangle \rmd \mu_{t_0} \Big )^{\frac{1}{2}},
~ \text{\rm for } T_1< t_0 <T_2,
\end{equation*}
with equality whenever $[w_{t_0}]\in {\rm Tan}_{J, \mu_{t_0}} \mathcal{P}(\R^{2d})$ and $\partial_t \mu_t+{\rm div} \big(J w_t \mu_t \big) =0$.
\end{lem}

We are now in a position to characterize the metric derivative of the curve $t \mapsto \mu_t$.

\begin{thm}\label{wjder}
Suppose that Assumption \ref{reg} holds. Then the metric derivative of the curve $t \mapsto \mu_t$ is given by
\begin{equation}\label{mder}
|\dot{\mu}_{t_0}|=\lim_{t \downarrow t_0} \dfrac{W_J(\mu_t, \mu_{t_0})}{t-t_0} = | w_{t_0} |_{J,\mu_{t_0}} =(\tilde{\I}(u_{t_0}))^{\frac{1}{2}}, ~ \text{\rm for } 0< t_0 <T.
\end{equation}
\end{thm}

\begin{proof}
Let
\begin{equation*}
m_t:= \Big( \sigma^2 \int_{\R^{2d}} \frac{|\nabla u_t|^2}{u_t}\rmd \tilde{z} + \int_{\R^{2d}} 2\beta \gamma |\tilde{v}|^2 u_t \rmd \tilde{z} \Big)^{\frac{1}{2}},~ t\in [0,T].
\end{equation*}
We obtain
\begin{equation*}
W_J(\mu_{t_1}, \mu_{t_2}) \leq \int_{t_1}^{t_2}  m_t  \rmd t < +\infty
\end{equation*}
by \eqref{estimate}.
This shows that $\mu$ is an absolutely continuous curve. By Lemma \ref{derivative}, it follows that $|\dot{\mu}_t| \leq | w_t |_{J, \mu_t}$.

Next, we prove that $[w_t]\in {\rm Tan}_{J, \mu_t} \mathcal{P}(\R^{2d})$.
By the definition of the equivalence relation $\sim$, we have $w_t \sim \nabla_{\tilde{v}} \ln u_t + \beta m \tilde{v}$.
Let $g_M\in C^{\infty}(\R)$ be such that
\begin{equation*}
g_M(\tau):=
\begin{cases}
 M+\frac{1}{2}, \quad & \tau >M+1, \\
 \bar{g}(\tau), &  M \leq |\tau| \leq M+1, \\
 \tau, & |\tau|<M, \\
 -M-\frac{1}{2},  & \tau<-M-1,
\end{cases}
\end{equation*}
where $M>0$ is a constant and $\bar{g}$ is a smooth function satisfying $0< \bar{g}'(\tau)<1$.
Define
$$h_M(\tilde{z}):=-g_M(\ln u_t(\tilde{z})) -\frac{1}{2}\beta m \tilde{v}^2.$$
Then $\nabla_{\tilde{v}} h_M (\tilde{z})= -g_M'(\ln u_t(\tilde{z})) \nabla_{\tilde{v}}\ln u_t(\tilde{z}) -\beta m \tilde{v}$ and
\begin{equation*}
\begin{aligned}
& | \nabla_{\tilde{v}} h_M (\tilde{z})- \nabla_{\tilde{v}} \ln u_t(\tilde{z}) - \beta m \tilde{v} | \\
=& | g_M'(\ln u_t(\tilde{z})) -1 | \cdot | \nabla_{\tilde{v}}\ln u_t(\tilde{z}) | \\
\leq & | \nabla_{\tilde{v}}\ln u_t(\tilde{z}) | \cdot 1_{\{ |\ln u_t|>M \}} \xrightarrow{\text{a.e.}} 0, ~ \text{ as } M \rightarrow +\infty.
\end{aligned}
\end{equation*}
Combining this with \eqref{int},
we deduce
\begin{equation*}
\begin{aligned}
\int_{\R^{2d}} | \nabla_{\tilde{v}} h_M (\tilde{z})- \nabla_{\tilde{v}} \ln u_t(\tilde{z}) - \beta m \tilde{v} |^2  u_t(\tilde{z}) \rmd \tilde{z} 
\leq  \int_{\R^{2d}} | \nabla_{\tilde{v}}\ln u_t(\tilde{z}) |^2 u_t(\tilde{z}) \rmd \tilde{z} <+\infty.
\end{aligned}
\end{equation*}
Therefore,
\begin{equation*}
\begin{aligned}
& | [\nabla h_M]- [w_t] |_{J, \mu_t}  \\
= & \sigma^2 \int_{\R^{2d}} | \nabla_{\tilde{v}} h_M (\tilde{z})- \nabla_{\tilde{v}} \ln u_t(\tilde{z}) - \beta m \tilde{v} |^2 u_t(\tilde{z}) \rmd \tilde{z}
\rightarrow 0, ~ \text{ as } M \rightarrow +\infty.
\end{aligned}
\end{equation*}
Choose $\chi_R\in C_c^{\infty}(\R^{2d})$ such that
\begin{equation*}
\chi_R(\tilde{z}):=
\begin{cases}
 1, \quad & |\tilde{z}| < R, \\
 \bar{\chi}(\tilde{z}), &  R\leq | \tilde{z}|< 2R , \\
 0, & | \tilde{z}| \geq 2R,
\end{cases}
\end{equation*}
where $\bar{\chi}$ is a smooth function satisfying $| \nabla \bar{\chi}(\tilde{z}) |< \frac{2}{R} $.
Set $\zeta_{M, R}(\tilde{z}):= \chi_R(\tilde{z}) h_M(\tilde{z})$.
Then $\nabla_{\tilde{v}} \zeta_{M, R}(\tilde{z})=\chi_{R}(\tilde{z}) \nabla_{\tilde{v}} h_M(\tilde{z}) + \nabla_{\tilde{v}} \chi_{R}(\tilde{z}) h_M(\tilde{z})$ and
\begin{equation}\label{xxx}
| \nabla_{\tilde{v}} \zeta_{M, R}(\tilde{z}) - \nabla_{\tilde{v}} h_M(\tilde{z})| \leq  | (\chi_{R}(\tilde{z})-1) \nabla_{\tilde{v}} h_M(\tilde{z})| + |\nabla_{\tilde{v}} \chi_{R}(\tilde{z}) h_M(\tilde{z}) |.
\end{equation}
For the first term on the right hand side of \eqref{xxx}, we have
\begin{equation*}
| (\chi_{R}(\tilde{z})-1) \nabla_{\tilde{v}} h_M(\tilde{z})| \xrightarrow{\text{a.e.}} 0,~ \text{ as } R \rightarrow +\infty.
\end{equation*}
The estimate \eqref{estimate} yields
\begin{equation*}
\int_{\R^{2d}} | (\chi_{R}(\tilde{z})-1) \nabla_{\tilde{v}} h_M(\tilde{z})|^2 u_t(\tilde{z}) \rmd \tilde{z} < +\infty.
\end{equation*}
Hence,
\begin{equation}\label{es1}
\int_{\R^{2d}} | (\chi_{R}(\tilde{z})-1) \nabla_{\tilde{v}} h_M(\tilde{z})|^2 u_t(\tilde{z}) \rmd \tilde{z} \rightarrow 0, ~ \text{ as } R \rightarrow +\infty.
\end{equation}
For the second term, it follows from \eqref{moment} that
\begin{equation}\label{es2}
\begin{aligned}
& \int_{\R^{2d}} |\nabla_{\tilde{v}} \chi_{R}(\tilde{z}) h_M(\tilde{z}) |^2 u_t(\tilde{z}) \rmd \tilde{z} \\
\leq & \dfrac{2}{R} \int_{\R^{2d}} (|M+1|^2 +\dfrac{1}{4}\beta^2 m^2 \tilde{v}^4 ) u_t(\tilde{z}) \rmd \tilde{z} \rightarrow 0, ~ \text{ as } R \rightarrow +\infty.
\end{aligned}
\end{equation}
Equations \eqref{es1} and \eqref{es2} therefore give
\begin{equation*}
\begin{aligned}
& | [\nabla \zeta_{M, R}]- [\nabla h_M] |_{J, \mu_t} \\
= &\sigma^2 \int_{\R^{2d}} | \nabla_{\tilde{v}} \zeta_{M, R}(\tilde{z}) - \nabla_{\tilde{v}} h_M(\tilde{z})|^2 u_t(\tilde{z}) \rmd \tilde{z} \\
\leq & 2\sigma^2 \int_{\R^{2d}} | (\chi_{R}(\tilde{z})-1) \nabla_{\tilde{v}} h_M(\tilde{z})|^2 u_t(\tilde{z}) \rmd \tilde{z} \\
& + 2\sigma^2 \int_{\R^{2d}} |\nabla_{\tilde{v}} \chi_{R}(\tilde{z}) h_M(\tilde{z}) |^2 u_t(\tilde{z}) \rmd \tilde{z}
\rightarrow 0, ~ \text{ as } R \rightarrow +\infty.
\end{aligned}
\end{equation*}

Using a diagonal argument, we can select a subsequence ${\zeta_n}$ with
$$\zeta_n(\tilde{z}):= \chi_{R(n)}(\tilde{z}) \big (-g_{M(n)} ( \ln u_t(\tilde{z})) -\frac{1}{2}\beta m \tilde{v}^2 \big)$$
such that
\begin{equation*}\label{zeta}
\begin{aligned}
 & | [\nabla \zeta_n]- [w_t] |_{J, \mu_t} \\
= &\sigma^2 \int_{\R^{2d}} | \nabla_{\tilde{v}} \zeta_n (\tilde{z})- \nabla_{\tilde{v}} \ln u_t(\tilde{z}) + \beta m \tilde{v} |^2  u_t(\tilde{z}) \rmd \tilde{z} \\
\leq & 2\sigma^2 \int_{\R^{2d}} | \nabla_{\tilde{v}} \zeta_n(\tilde{z}) - \nabla_{\tilde{v}} h_{M(n)}(\tilde{z})|^2 u_t(\tilde{z}) \rmd \tilde{z} \\
 & + 2\sigma^2
\int_{\R^{2d}} | \nabla_{\tilde{v}} h_{M(n)} (\tilde{z})- \nabla_{\tilde{v}} \ln u_t(\tilde{z}) - \beta m \tilde{v} |^2 u_t(\tilde{z}) \rmd \tilde{z}
\rightarrow 0, ~ \text{ as } n \rightarrow +\infty,
\end{aligned}
\end{equation*}
which shows that $[w_t]\in {\rm Tan}_{J, \mu_t} \mathcal{P}(\R^{2d})$. Applying Lemma \ref{derivative} then yields \eqref{mder}.
\end{proof}

{\bf Step 3: Metric slope of the free energy and gradient flow.}
Now, we establish the $W_J$-gradient flow formulation.
Our approach is to show that the curve
$t \mapsto \mu_t$ evolves in the direction of steepest descent of the free energy $\tilde{\F}$.
In analogy with the classical $W_2$-gradient flow theory in \cite{AGS}, we express the $W_J$-gradient flow at $\mu_0$ with form
\begin{equation}\label{grad}
{\rm grad}_{W_J} \tilde{\F}_0(\mu_0)=-{\rm div}\Big(J \mu_0 \nabla\big(\dfrac{\delta \tilde{\F}_0(\mu_0) }{\delta \mu}\big)\Big).
\end{equation}
Based on this formulations, we obtain the following result.

\begin{thm}
Suppose that Assumption \ref{reg} holds.
Then the gradient flow of the free energy functional $\tilde{\F}$ with respect to $W_J$ is given by
\begin{equation}\label{grad2}
\begin{aligned}
\partial_t u_t= & -{\rm grad}_{W_J} \tilde{\F}_t(u_t) \\
= & {\rm div}\Big( J u_t \nabla \big( \ln u_t+ \beta m h_{f_t}(\Phi_t(\tilde{z})) \big) \Big) \\
= & -\J(u_t) \big( \ln u_t+ \beta m h_{f_t}(\Phi_t(\tilde{z})) \big).
\end{aligned}
\end{equation}
\end{thm}

\begin{proof}
We use the notation
\begin{equation*}
\begin{aligned}
\big|\partial \tilde{\F}_{t_0} \big|_{W_J}(\mu_{t_0})
:=  \lim_{t \downarrow t_0} \dfrac{ \tilde{\F}_t(\mu_t) -\tilde{\F}_{t_0}(\mu_{t_0}) }{W_J\big(\mu_t, \mu_{t_0}\big)}
\end{aligned}
\end{equation*}
to denote the Otto-Onsager metric slope of the free energy functional $\tilde{\F}$ along the curve $t \mapsto \mu_t$ with respect to $W_J$.

According to Theorems \ref{fder} and \ref{wjder}, we obtain
\begin{equation*}\label{slope1}
\begin{aligned}
\big|\partial \tilde{\F}_{t_0} \big|_{W_J}(\mu_{t_0})
= & - \Big(\dfrac{\gamma}{\beta m^2} \int_{\R^{2d}} | \nabla_{\tilde{v}} \ln u_{t_0} + \beta m \tilde{v} |^2 u_{t_0} \rmd \tilde{z} \Big)^{\frac{1}{2}} \\
= & - \sqrt{\dfrac{\gamma}{\beta m^2}} \big \| \nabla_{\tilde{v}} \ln u_{t_0} + \beta m \tilde{v} \big \|_{L^2(\mu_{t_0})}.
\end{aligned}
\end{equation*}

Fix $t_0\in(0, T) $. We introduce a perturbation
$\varepsilon \in C_c^{\infty}(\R^{2d}, \R)$ and set
$$b_D^{\varepsilon}:= -\frac{\beta m}{2} \nabla \big(h_f + \varepsilon \big) A=b_D -\frac{\gamma }{m} \binom{0}{\nabla_{v} \varepsilon}.$$
The corresponding perturbed systems are then given by
\begin{equation*}\label{FPEp}
\begin{aligned}
\partial_t u_t^{\varepsilon}(\tilde{z}) &= -\partial_i(\tilde{b}^{i,\varepsilon}(\tilde{z}) u_t^{\varepsilon}(\tilde{z}) ) +\dfrac{1}{2}\partial_{ij}((\tilde{G}\tilde{G}^{\top})^{ij}(\tilde{z}) u_t^{\varepsilon}(\tilde{z})),  \\
u^{\varepsilon}_{t_0}(\tilde{z})&=u_{t_0}(\tilde{z})
\end{aligned}
\end{equation*}
and
\begin{equation*}\label{SDEp}
\begin{aligned}
\rmd \tilde{Z}_t^{\varepsilon} &=\tilde{b}^{\varepsilon}(\tilde{Z}_t^{\varepsilon})\rmd t+\tilde{G}(\tilde{Z}_t^{\varepsilon})\rmd B_t, \\
\tilde{Z}^{\varepsilon}_{t_0}&=\tilde{Z}_{t_0},
\end{aligned}
\end{equation*}
where
$\tilde{b}^{i,\varepsilon}(\tilde{z}):=[ D\Phi_{-t}(\Phi_t(\tilde{z}) b_D^{\varepsilon}(\Phi_t(\tilde{z})) ]^i +\frac{1}{2}\tr (A D^2\Phi_{-t}^i(\Phi_t(\tilde{z})) )$.
We define $\mu^{\varepsilon}_t$ accordingly.
We choose the perturbation $-\frac{\gamma }{m} \binom{0}{\nabla_{v} \varepsilon}$ so that the perturbed field $b_D^{\varepsilon}$ preserves the key structural properties of the unperturbed system.
In particular, this choice maintains the Gibbs state structure and ensures that
\begin{equation*}
(v+ \nabla_{v} \varepsilon) \cdot v \geq |v|^2 - |v| | \nabla_{v} \varepsilon |
\geq \frac{1}{2}|v|^2 -\frac{1}{2} \| \nabla_{v} \varepsilon \|^2_{\infty}.
\end{equation*}

Next, we compute the Otto-Onsager metric slope of the free energy functional $\tilde{\F}$ along the perturbed curve $t \mapsto \mu^{\varepsilon}_t$ with respect to $W_J$.
Throughout this computation, it is important to note that the free energy functional under consideration remains
$$\tilde{\F}_t(u)=\int_{\R^{2d}} u \ln u + \beta m u \check{h}_{f_t}((\Phi_t(\tilde{z}))) \rmd \tilde{z},$$
rather than $\int_{\R^{2d}} u \ln u + \beta m u \check{h}_{f_t}((\Phi_t(\tilde{z}))) +\beta m u \varepsilon((\Phi_t(\tilde{z}))) \rmd \tilde{z}$.
Proceeding as in the proofs of Theorems \ref{fder} and \ref{wjder}, we get that the free energy dissipation of
$\tilde{\F}$ along $t \mapsto \mu^{\varepsilon}_t$ is given by
\begin{equation*}
\begin{aligned}
\dfrac{\rmd}{\rmd t}\tilde{\F}_t(u_t^{\varepsilon} )
= &-\dfrac{\gamma}{\beta m^2} \int_{\R^{2d}} \big \langle \nabla_{\tilde{v}} \ln u_t^{\varepsilon} + \beta m \tilde{v} , \nabla_{\tilde{v}} \ln u_t^{\varepsilon} + \beta m \tilde{v} +\beta m \nabla_{\tilde{v}} \varepsilon(\Phi_t(\tilde{z})) \big \rangle u_t^{\varepsilon}  \rmd \tilde{z} \\
= &-\dfrac{\gamma}{\beta m^2} \big \langle \nabla_{\tilde{v}} \ln u_t^{\varepsilon} + \beta m \tilde{v} ,
\nabla_{\tilde{v}} \ln u_t^{\varepsilon}  +  \beta m \tilde{v} +\beta m \nabla_{\tilde{v}} \varepsilon(\Phi_t(\tilde{z})) \big \rangle_{L^2(\mu^{\varepsilon}_{t})}
\end{aligned}
\end{equation*}
and the metric derivative satisfies
\begin{equation*}
\begin{aligned}
|\dot{\mu}^{\varepsilon}_{t_0}|= &\lim_{t \downarrow t_0} \dfrac{W_J(\mu^{\varepsilon}_t, \mu^{\varepsilon}_{t_0})}{t-t_0} \\
= & \Big(\dfrac{\gamma}{\beta m^2} \int_{\R^{2d}} \big| \nabla_{\tilde{v}} \ln u_{t_0}^{\varepsilon} + \beta m \tilde{v} +\beta m \nabla_{\tilde{v}} \varepsilon(\Phi_t(\tilde{z})) \big |^2 u_{t_0}^{\varepsilon}  \rmd \tilde{z} \Big)^{\frac{1}{2}} \\
=& \sqrt{\dfrac{\gamma}{\beta m^2}} \big \|  \nabla_{\tilde{v}} \ln u_{t_0}^{\varepsilon} +  \beta m \tilde{v} +\beta m \nabla_{\tilde{v}} \varepsilon(\Phi_t(\tilde{z}))  \big \|_{L^2(\mu_{t_0}^{\varepsilon})} .
\end{aligned}
\end{equation*}
Consequently, we obtain
\begin{equation*}\label{slope2}
\begin{aligned}
& \big|\partial \tilde{\F}_{t_0} \big|_{W_J}(\mu^{\varepsilon}_{t_0}) \\
= & -\sqrt{\dfrac{\gamma}{\beta m^2}} \Big \langle \nabla_{\tilde{v}} \ln u_{t_0}^{\varepsilon} + \beta m \tilde{v} ,
\dfrac{\nabla_{\tilde{v}} \ln u_{t_0}^{\varepsilon}  +  \beta m \tilde{v} +\beta m \nabla_{\tilde{v}} \varepsilon(\Phi_t(\tilde{z}))}{ \|  \nabla_{\tilde{v}} \ln u_{t_0}^{\varepsilon}  +  \beta m \tilde{v} +\beta m \nabla_{\tilde{v}} \varepsilon(\Phi_t(\tilde{z}))  \|_{L^2(\mu_{t_0}^{\varepsilon})} } \Big \rangle_{L^2(\mu^{\varepsilon}_{t_0})}.
\end{aligned}
\end{equation*}

Since $\mu^{\varepsilon}_{t_0}=\mu_{t_0}$,
the Cauchy-Schwarz inequality yields
\begin{equation*}
\big|\partial \tilde{\F}_{t_0} \big|_{W_J}(\mu^{\varepsilon}_{t_0})  \geq \big|\partial \tilde{\F}_{t_0} \big|_{W_J}(\mu_{t_0}),
\end{equation*}
with equality if and only if $b_D^{\varepsilon}= C_{\varepsilon} b_D$ for some constant $C_{\varepsilon}>0$.
Since the perturbation $\varepsilon$ is arbitrary, it follows that, at time $t_0$, the curve $t \mapsto \mu_t$ minimizes the Otto-Onsager metric slope. Because
$t_0\in (0, T)$ is arbitrary, we conclude that $t \mapsto \mu_t$ evolves in the direction of steepest descent of the free energy $\tilde{\F}$.
This proves \eqref{grad2}.
\end{proof}

\begin{rem}
The Wasserstein gradient flow representation of \eqref{VFPE2} is time-inhomogeneous, due to the pull-back structure of the underlying system and its explicit time dependence. We therefore refer to it as a generalized Wasserstein gradient flow.
\end{rem}


\subsection{GENERIC}\label{generic}

We now return to the original system \eqref{VFPE} through the flow $\Phi$.
Combining
\begin{equation*}
\begin{aligned}
\partial_t u_t(\tilde{z})=& \big (\partial_t f_t(z) + \nabla f_t(z) \cdot b_H(z) \big ) \big|_{z=\Phi_t(\tilde{z})} \\
=& \big (\partial_t f_t(z) + \nabla (f_t(z) b_H(z)) \big ) \big|_{z=\Phi_t(\tilde{z})}
\end{aligned}
\end{equation*}
and
\begin{equation*}
\begin{aligned}
 \partial_t u_t(\tilde{z})= & {\rm div}_{\tilde{z}}\big(\frac{1}{2}\tilde{A}(\tilde{z}) (\nabla_{\tilde{z}} \ln u_t(\tilde{z}) +\beta m \nabla_{\tilde{z}} h_{f_t}(\tilde{z})) u_t(\tilde{z}) \big)\\
= & {\rm div}_z \big( \frac{1}{2} A  ( \nabla_z \ln f_t(z) +\beta m \nabla_z h_{f_t}(z)) f_t(z) \big) \big|_{z=\Phi_t(\tilde{z})},
\end{aligned}
\end{equation*}
we obtain
\begin{equation*}
\begin{aligned}
\partial_t f_t = & L_{rev}(f_t)+ L_{irr}(f_t) \\
=& -\{ f_t, h_{f_t} \}+ {\rm div}\Big (\frac{1}{2}A f_t \nabla(\ln f_t +\beta m h_{f_t}) \Big) \\
= & \L(f_t) \dfrac{\delta H_{f}}{\delta f}(f_t) + \J (f_t) \dfrac{\delta \F}{\delta f}(f_t).
\end{aligned}
\end{equation*}
This expression coincides with the form in \eqref{gen}.
This demonstrates that the Vlasov-Fokker-Planck equation \eqref{VFPE} does not admit a classical gradient flow formulation.
Instead, it possesses a GENERIC structure.
The underlying reason is that the evolution is governed not only by dissipation but also by a reversible component, represented by $L_{rev}(f_t)$.
Nevertheless, after an appropriate reparametrization, the system can still be interpreted as a generalized gradient flow with respect to $W_J$.
This gradient flow viewpoint yields a more refined description of the dynamics.

\begin{thm}\label{ori}
Suppose that Assumption \ref{reg} holds. Then, for $t_0\in[0,T)$, the trajectorial rate of free energy dissipation identity of \eqref{MFLE} is given by
\begin{equation}\label{cortra}
\begin{aligned}
& \lim_{t \downarrow t_0} \dfrac{\mathbb{E}[\theta_{t}(Z_{t}, \rho_{t})| \mathcal{G}_{t_0}] -\theta_{t_0}(Z_{t_0}, \rho_{t_0}) }{t-t_0} \\
= &\mathcal{D}_{t_0}(Z_{t_0}) \\
:= &2\dfrac{\gamma}{m} +\dfrac{\sigma^2}{2 f_{t_0}(Z_{t_0})} \Delta_v f_{t_0}(Z_{t_0}) -\beta\gamma |V_{t_0}|^2 +\dfrac{1}{2}\sigma^2 \Delta_v \ln f_{t_0}(Z_{t_0}) -\dfrac{1}{2}\beta \big \langle \nabla K*f_{t_0}(X_{t_0}), V_{t_0} \big \rangle \\
 & + \E_{\bar{\P}} \Big[ \Big\langle \dfrac{1}{2}\beta \nabla K(X_{t_0}-Y_{t_0}^1), V_{t_0} \Big \rangle  \Big] ,~ \P-a.s.,
\end{aligned}
\end{equation}
where $\theta_t(z, f_t)=\ln f_t(z) + \beta m f_t \check{h}_{f_t}(z)$ and $Y=(Y^1, Y^2)^{\top}$ denotes an independent copy of $Z$ defined on $(\bar{\Omega}, \bar{\mathcal{G}}, \bar{\P})$.
Moreover, the free energy dissipation of \eqref{VFPE} is given by
\begin{equation}\label{coraver}
\dfrac{\rmd}{\rmd t}\F(f_t)= -\dfrac{1}{2} \int_{\R^{2d}} | \nabla (\ln f_t+\beta m \check{h}_{f_t} ) G |^2 f_t \rmd z=-\mathcal{I}(f_t) > -\infty.
\end{equation}
\end{thm}

\begin{proof}
According to Remark \ref{hrem} (ii), the pullback preserves volume.
Therefore, the required integrability holds automatically, and we omit the details.
According to \eqref{D}, we have
\begin{equation*}
\begin{aligned}
 \mathcal{D}_{t_0}(Z_{t_0})
=&2\dfrac{\gamma}{m} +\dfrac{\sigma^2}{2 f_{t_0}(Z_{t_0})} \Delta_v f_{t_0}(Z_{t_0}) -\beta\gamma |V_{t_0}|^2 +\dfrac{1}{2}\sigma^2 \Delta_v \ln f_{t_0}(Z_{t_0})  \\
 & -\dfrac{1}{2}\beta \big \langle \nabla K*f_{t_0}(X_{t_0}), V_{t_0} \big \rangle
 + \E_{\bar{\P}} \Big[ \Big\langle \dfrac{1}{2}\beta \nabla K(X_{t_0}-Y_{t_0}^1), V_{t_0} \Big \rangle  \Big],
\end{aligned}
\end{equation*}
which establishes \eqref{cortra}.
Substituting $\tilde{z}=\Phi_{-t}(z)$ into \eqref{average} yields \eqref{coraver}.
\end{proof}

\section{The partial HWI inequality}\label{degen}
In this section, we highlight the role of partial degeneracy.
We specialize to the case $\nabla K=0$ and perform a detailed study of the induced metric structure.
This setting makes it possible to isolate more clearly the impact of partial degeneracy on the metric $W_J$.
On this basis, we derive a partial HWI inequality.

As noted in Subsection \ref{GF}, $W_J$ is only a quasi-metric, since it may take the value $+\infty$.
Applying the disintegration theorem to $\mu_t$, we obtain a measurable family of probability measures $\{ \mu_{t}^x \}_{x\in \R^d} \in \mathcal{P}(\R^{d})$ such that
$$\mu_t(\rmd x \rmd v)=\bar{\mu}_{t}(\rmd x) \mu_t^x (\rmd v)$$
and
\begin{equation*}
\bar{\mu}_{t}( \rmd x)=\int_{\R^d} \mu_t(\rmd x, v) \rmd v.
\end{equation*}
As before, we write $\bar{u}_{t}$ and $u_t^x$ for the densities of $\bar{\mu}_{t}$ and $\mu_t^x$, respectively.
In the partially degenerate case, we may take
$J:=\left(\begin{smallmatrix} 0 & 0 \\ 0 & M \end{smallmatrix}\right)$,
where $M$ is a symmetric positive definite matrix.
Then
\begin{equation*}
\partial_t u_t+ {\rm div}_{v}(M u_t \nabla_{v} \zeta)=0
\end{equation*}
and
\begin{equation*}
\partial_t \bar{u}_{t} =\int_{\R^d} \partial_t u_t \rmd v= -\int_{\R^d} {\rm div}_{v}(M u_t \nabla_{v} \zeta) \rmd v =0,
\end{equation*}
where $\zeta \in C^{\infty}(\R^{2d})$.
That is, the curves arising from the continuity equation share the same $x$-marginal.
Consequently, for a curve $\mu:(T_1, T_2)\rightarrow \mathcal{P}(\R^{2d})$, if there exist $T_1 < t_1<t_2<T_2$ such that $\bar{\mu}_{t_1}\neq \bar{\mu}_{t_2}$, then $W_J(\mu_{t_1}, \mu_{t_2})=+\infty$.
Therefore, any absolutely continuous curve must remain within an appropriate submanifold
\begin{equation*}
\mathcal{P}_{\nu}(\R^{2d}):=\Big\{ \mu \in \mathcal{P}(\R^{2d}) \Big | \int_{\R^d} \mu \rmd v= \int_{\R^d} \nu\rmd v \Big\}.
\end{equation*}
This observation also clarifies the role of the pullback of \eqref{VFPE}.
For \eqref{CE}, we have
\begin{equation*}
\partial_t \bar{u}_{t} =\int_{\R^d} \partial_t u_t \rmd \tilde{v}= -\int_{\R^d} {\rm div}_{\tilde{v}}(M u_t \nabla_{\tilde{v}} \tilde{\theta}_t) \rmd \tilde{v} =0.
\end{equation*}
Hence $\mu_t \in \mathcal{P}_{\mu_0}(\R^{2d})$ for all $t\in[0,T]$.
From this point onward, we restrict the metric $W_J$ to the submanifold $\mathcal{P}_{\mu_0}(\R^{2d})$ and work within this setting.
Fix $x\in \R^d$. We select the corresponding fiber $\mathcal{P}^x_{\mu_0}(\R^{d})$ and define on it the metric
\begin{equation*}
\begin{aligned}
W_x^2(\mu^x_0, \mu^x_1):=\inf_{(\mu^x_t, w^x_t)} \Big\{
& \int_0^1 \int_{\R^d} \langle w^{x}_t, M w^x_t \rangle \rmd \mu^x_t \rmd t \bigg|  \partial_t \mu^x_t+{\rm div} \big(M w^x_t \mu^x_t \big) =0 \\
 & \text{ holds in the distributional sence} \Big\}.
\end{aligned}
\end{equation*}
If we aggregate the ``cost" on each fiber $\mathcal{P}^x_{\mu_0}(\R^{d})$ weighted by $\bar{\mu}_{0}$, we obtain
\begin{equation*}
W_J^2(\mu_0, \mu_1)=\int_{\R^d} W_x^2 (\mu_0^x, \mu_1^x) \rmd \bar{\mu}_{0} = \int_{\R^d} \int_{\R^d} \langle \nabla \psi, M \nabla \psi\rangle \rmd \mu_0^x \rmd \bar{\mu}_{0}.
\end{equation*}

This analysis shows that the space $(\mathcal{P}(\R^{2d}), W_J )$ has a decomposed structure.
The distance between points in different pieces is infinite.
On each piece, the distance admits an integral representation in terms of the corresponding fiberwise distances.

Next, we exploit this structure to derive the partial HWI inequality.
We define the relative entropy of $\mu$ with respect to $\mu_{\infty}$ by
\begin{equation}\label{rentrop}
H(\mu | \mu_{\infty}):= \int_{\R^{2d}} u\ln \frac{u}{u_{\infty}} \rmd z
\end{equation}
and the partially relative Fisher information of $\mu$ with respect to $\mu_{\infty}$ by
\begin{equation}\label{fisher}
I(\mu | \mu_{\infty}):
= \int_{\R^{2d}} \big \langle \nabla_v \ln \frac{u}{u_{\infty}}, M \nabla_v \ln \frac{u}{u_{\infty}} \big \rangle u \rmd z,
\end{equation}
where $\rmd \mu=u \rmd z$ and $\rmd \mu_{\infty}=u_{\infty} \rmd z$.
If $\mu$ and $\mu_{\infty}$ are not absolutely continuous, we set the value to $+\infty$.

To establish the partial HWI inequality, we first ensure that the relative entropy is displacement convex.
For this purpose, we impose the following assumption.

\begin{ass}\label{hess}
There exists a constant $\kappa_M$ such that the matrix $M$ satisfies
$${\rm Ric}_{M^{-1}(x, \cdot)} -{\rm Hess}_{M^{-1}(x, \cdot)}\big(\dfrac{1}{2}\ln |M(x, \cdot)| +\ln u_{\infty}(x, \cdot) \big) \geq \kappa_M {\rm I}_{d \times d}, ~ \text{for a.e. }x.$$
\end{ass}

\begin{thm}
Suppose that Assumption \ref{hess} holds.
Let $\mu_0, \mu_1  \in \mathcal{P}_{\mu_0}(\R^{2d})$ be absolutely continuous probability measures.
If the relative entropy $H(\mu_0| \mu_\infty)$ is finite, then
\begin{equation}\label{HWI}
H(\mu_0 | \mu_{\infty})- H(\mu_1 | \mu_{\infty}) \leq \sqrt{I(\mu_0 | \mu_{\infty})}~ W_J(\mu_0, \mu_1) - \dfrac{\kappa_M}{2}W_J^2(\mu_0, \mu_1).
\end{equation}
\end{thm}

\begin{proof}
When $I(\mu_0 | \mu_{\infty})=+\infty$ or $H(\mu_1 | \mu_{\infty})=+\infty$, \eqref{HWI} holds trivially. We therefore restrict attention to the nontrivial case.

We denote the density functions of $\mu_0$, $\mu_1$ and $\mu_{\infty}$ by $u_1$, $u_2$ and $u_{\infty}$, respectively.
Exploiting the structure of the metric $W_J$, we first disintegrate the measures $\mu_0$ and $\mu_1$ and reduce the analysis to the fiberwise metric $W_x$.
In this setting, $W_x$ can be regarded as a metric on the Riemannian manifold $\R^d$ endowed with the metric tensor $\bm{g}=M^{-1}$.
Consequently, $W_x$ admits the representation
\begin{equation}\label{distx}
\begin{aligned}
W_x^2(\mu^x_0, \mu^x_1)=\inf_{(\mu_t^x, \bm{w}^x_t)} \Big\{
& \int_0^1 \int_{\R^d} |\bm{w}_t^x|^2_{\bm{g}} \rmd \mu^x_t \rmd t \bigg|  \partial_t \mu^x_t+{\rm div} \big( \bm{w}^x_t \mu^x_t \big) =0 \\
 & \text{ holds in the distributional sence} \Big\},
\end{aligned}
\end{equation}
where $|\xi|_{\bm{g}}^2:= \xi^{\top} M^{-1} \xi$ and $\bm{w}_t^x=M w_t^x$.

Let
\begin{equation*}
\begin{aligned}
{\rm Adm} (\mu^x_0, \mu^x_1):=\{& \pi\in \mathcal{P}(\R^{d}\times \R^d) \big| \Pi^1_{\sharp}\pi=\mu^x_0,  \Pi^2_{\sharp}\pi=\mu^x_1, \\
& \text{ where } \Pi^1, \Pi^2 \text{ are the natural projections} \}.
\end{aligned}
\end{equation*}
Theorem 3.30 in \cite{AG} shows that \eqref{distx} is equivalent to the associated ``static" optimal transport problem, namely,
\begin{equation*}
W_x^2(\mu^x_0, \mu^x_1)=\inf_{\pi\in {\rm Adm} (\mu^x_0, \mu^x_1)} \int_{\R^d \times \R^d} d_{\bm{g}}^2(v_1, v_2) \rmd \pi,
\end{equation*}
where $d_{\bm{g}}$ denotes the distance induced by $\bm{g}$.
By Theorem 4.3 in \cite{FF}, the optimal transport map $T^x$ exists and is unique. It is given by
\begin{equation}\label{omap}
T^x(v):=\exp^x_v(-\nabla_{\bm{g}} \psi(v))=\exp^x_v(-M \nabla \psi(v)),
\end{equation}
where $\psi$ is a $d^2_{\bm{g}}$-convex function. Consequently,
\begin{equation*}
W_x^2(\mu^x_0, \mu^x_1)=\int_{\R^d} d_{\bm{g}}^2(v, T^x(v)) \rmd \mu_0^x.
\end{equation*}
The associated displacement interpolation is then given by
$$T^x_t(v):=\exp^x_v(-t \nabla_{\bm{g}} \psi(v))=\exp^x_v(- t M \nabla \psi(v)), ~\mu_t^x:=(T_t^x)_{\sharp}\mu_0^x.$$
Based on the optimal transport map $T^x$, we define a velocity field $\bm{w}_t^x$ by requiring that
\begin{equation}\label{ofield}
\dot{T^x_t}(v)=\bm{w}_t^x(T^x_t(v)).
\end{equation}
If $V_t$ is a random variable with law $\mu_t^x$, this yields
\begin{equation*}
\rmd T^x_t(V_t)=\bm{w}_t^x(T^x_t(V_t)).
\end{equation*}
Since $J$ degenerates in the $x$-direction, for a random variable $X_t$ with law $\bar{u}_{t}$, we have
\begin{equation*}
\rmd X_t=0.
\end{equation*}

According to
\begin{equation*}
\begin{aligned}
\rmd \mu_t^x(V_t)= & \partial_t \mu_t^x(V_t) \rmd t +\langle \nabla \mu_t^x(V_t), \rmd V_t \rangle \\
= & -{\rm div}_v(\mu_t^x(V_t) \bm{w}_t^x(V_t)) \rmd t+\langle \nabla \mu_t^x(V_t), \rmd V_t \rangle \\
= & -\mu_t^x(V_t) {\rm div}_v(\bm{w}_t^x(V_t)) \rmd t,
\end{aligned}
\end{equation*}
we have
\begin{equation*}
\begin{aligned}
\rmd \theta(u_t(X_t, V_t)) := & \rmd \big(\ln u_t(X_t, V_t) -\ln u_{\infty}(X_t, V_t) \big)\\
= & \langle \nabla_x (\ln u_t -\ln u_{\infty})(X_t, V_t), \rmd X_t\rangle + \langle \nabla_v (\ln u_t -\ln u_{\infty})(X_t, V_t), \rmd V_t \rangle \\
=& \langle \nabla_v \ln u_t(X_t, V_t)- \nabla_v \ln u_{\infty}(X_t, V_t),  \bm{w}_t^x(V_t)\rangle \rmd t.
\end{aligned}
\end{equation*}
Writing $H(\mu_t |\mu_{\infty})=\E[\theta(\mu_t)]$, we get that for any $0\leq t_1 <t_2 \leq 1$,
\begin{equation*}
\begin{aligned}
& H(\mu_{t_1} |\mu_{\infty}) -H(\mu_{t_2} |\mu_{\infty}) \\
= & \int_{t_2}^{t_1} \int_{\R^{2d}} \langle \nabla_v \ln u_t(x,v) - \nabla_v \ln u_{\infty}(x, v), \bm{w}_t^x(v) \big \rangle u_t(x,v) \rmd x\rmd v\rmd t \\
= & \int_{t_2}^{t_1} \int_{\R^{2d}} \big \langle \nabla_v \theta( u_t(x,v)), \bm{w}_t^x(v) \big \rangle u_t(x,v) \rmd x\rmd v\rmd t.
\end{aligned}
\end{equation*}

A direct calculation from \eqref{omap} gives
\begin{equation*}
\dot{T_t^x}(v)=\frac{\rmd}{\rmd t} \exp^x_v (-t\nabla_{\bm{g}} \psi(v))=-(\rmd \exp^x_v)_{-t\nabla_{\bm{g}} \psi(v)}(\nabla_{\bm{g}} \psi(v)).
\end{equation*}
Together with \eqref{ofield}, this implies
$$\bm{w}_0^x(v)=\dot{T_0^x}(v)=-\nabla_{\bm{g}} \psi(v)=-M \nabla \psi(v).$$
Consequently,
\begin{equation}\label{hd}
\dfrac{\rmd}{\rmd t} H(\mu_{t} |\mu_{\infty}) \Big |_{t=0^+} = \int_{\R^{2d}} \big \langle \nabla_v \theta( u_0(x,v)), -M \nabla \psi(v) \big \rangle u_0(x,v) \rmd x\rmd v.
\end{equation}

According to $\rmd \mu_t^x=u_t^x \rmd v=u_t^x \sqrt{|M|} \rmd {\rm vol}_{\bm{g}} =:\bm{u}_t^x \rmd {\rm vol}_{\bm{g}}$, we obtain
\begin{equation*}
\int_{\R^d} u_t^x \ln u_t^x \rmd v =\int_{\R^d} \bm{u}_t^x \ln \bm{u}_t^x \rmd {\rm vol}_{\bm{g}} -\dfrac{1}{2} \int_{\R^{d}} \bm{u}_t^x \ln |M| \rmd {\rm vol}_{\bm{g}}.
\end{equation*}
Therefore,
\begin{equation*}
\begin{aligned}
& H(\mu_t|\mu_{\infty}) \\
=& \int_{\R^{2d}} u_t \ln u_t - u_t \ln u_{\infty} \rmd z \\
=& \int_{\R^d}\int_{\R^d} \bar{u}_{t} u_t^x \ln (\bar{u}_{t} u_t^x ) - \bar{u}_{t} u_t^x \ln u_{\infty} \rmd v \rmd x  \\
= &\int_{\R^d} \bar{u}_{t} \Big (\int_{\R^d} u_t^x \ln u_t^x \rmd v \Big ) \rmd x+ \int_{\R^d} \bar{u}_{t}  \ln \bar{u}_{t} \rmd x
 - \int_{\R^d} \bar{u}_{t} \Big (\int_{\R^d} u_t^x \ln u_{\infty} \rmd v \Big ) \rmd x \\
=& \int_{\R^d} \bar{u}_{t} \Big (\int_{\R^d} \big(\bm{u}_t^x \ln \bm{u}_t^x -\dfrac{1}{2}\bm{u}_t^x \ln |M| -  \bm{u}_t^x \ln u_{\infty} \big) \rmd {\rm vol}_{\bm{g}} \Big ) \rmd x
 + \int_{\R^d} \bar{u}_{t}  \ln \bar{u}_{t} \rmd x .
\end{aligned}
\end{equation*}
Let
$$h_1(\bm{u}_t^x):=\int_{\R^d} \big(\bm{u}_t^x \ln \bm{u}_t^x -\frac{1}{2}\bm{u}_t^x \ln |M| -  \bm{u}_t^x \ln u_{\infty} \big) \rmd {\rm vol}_{\bm{g}},
~ h_2(\bar{u}_{t}):=\bar{u}_{t}  \ln \bar{u}_{t}.$$ Then
$$H(\mu_t|\mu_{\infty})=\int_{\R^d} \bar{u}_{t} h_1(\bm{u}_t^x)\rmd x +\int_{\R^d} h_2(\bar{u}_{t}) \rmd x.$$
The fact that the measure $\mu_t$ has a time independent $x$-marginal implies that $\frac{\rmd^2}{\rmd t^2}h_2(\bar{u}_{t})=0$.
Theorem 1.3 in \cite{S2} guarantees that, under Assumption \ref{hess}, $h_1(\bm{u}_t^x)$ is $\kappa_M$-convex.
Consequently,
\begin{equation*}
\dfrac{\rmd^2}{\rmd t^2} h_1(\bm{u}_t^x) \geq \kappa_M W_x^2(\mu_0^x, \mu_1^x)
\end{equation*}
by Theorem 4.1 in \cite{S2}.
Combining this with $\partial_t \bar{u}_t=0$, we deduce
\begin{equation}\label{2}
\begin{aligned}
\dfrac{\rmd^2}{\rmd t^2} H(\mu_t|\mu_{\infty})
= &\int_{\R^d} \bar{u}_{t} \dfrac{\rmd^2}{\rmd t^2} h_1(\bm{u}_t^x) + h_1(\bm{u}_t^x) \dfrac{\rmd^2}{\rmd t^2}\bar{u}_{t} \rmd x \\
\geq & \int_{\R^d} \bar{u}_{t}(x) \kappa_M W_x^2(\mu_0^x, \mu_1^x) \rmd x \\
=& \kappa_M \int_{\R^d} \bar{u}_{0}(x) W_x^2(\mu_0^x, \mu_1^x) \rmd x \\
=& \kappa_M W_J^2(\mu_0, \mu_1).
\end{aligned}
\end{equation}

By Taylor's formula
\begin{equation*}
\mathrm{F}(\mu_1) \geq \mathrm{F}(\mu_0) + \mathrm{F}'(\mu_0) +\int_0^1 (1-t) \mathrm{F}''(\mu_t)\rmd t,
\end{equation*}
\eqref{hd} and \eqref{2}, we conclude
\begin{equation}\label{taylor}
\begin{aligned}
& H(\mu_{0} |\mu_{\infty}) -H(\mu_{1} |\mu_{\infty}) \\
\leq & -\dfrac{\rmd}{\rmd t} H(\mu_{t} |\mu_{\infty}) \Big |_{t=0^+} - \dfrac{\kappa_M}{2}W_J^2(\mu_0, \mu_1) \\
= & \int_{\R^{2d}} \langle \nabla_v \theta( u_0(x,v)), M \nabla \psi(v) \rangle u_0(x,v) \rmd x\rmd v - \dfrac{\kappa_M}{2}W_J^2(\mu_0, \mu_1) \\
\leq & \Big( \int_{\R^{2d}} \big\langle \nabla_v \theta( u_0(x,v)), M \nabla_v \theta( u_0(x,v)) \big\rangle u_0(x,v) \rmd x\rmd v \Big)^{\frac{1}{2}} \\
& \times \Big( \int_{\R^{2d}} \big\langle \nabla \psi(v), M \nabla \psi(v) \big\rangle u_0(x,v) \rmd x\rmd v \Big)^{\frac{1}{2}} - \dfrac{\kappa_M}{2}W_J^2(\mu_0, \mu_1).
\end{aligned}
\end{equation}

Since $|\dot{T_t^x}|_{\bm{g}}=|\dot{T_0^x}|_{\bm{g}}$, we have
$d_{\bm{g}}(v, T^x v)=|\nabla_{\bm{g}} \psi(v)|_{\bm{g}}$, which implies
\begin{equation*}
W_x^2(\mu_0^x, \mu_1^x)=\int_{\R^d}  |\nabla_{\bm{g}} \psi(v)|_{\bm{g}}^2 \rmd \mu_0^x= \int_{\R^d} \langle \nabla \psi, M \nabla \psi \rangle \rmd \mu_0^x.
\end{equation*}
Hence,
\begin{equation}\label{om}
\begin{aligned}
& \int_{\R^{2d}} \big\langle \nabla \psi(v), M \nabla \psi(v) \big\rangle u_0(x,v) \rmd x\rmd v \\
= &\int_{\R^d} \bar{u}_{0}(x) \int_{\R^d} \langle \nabla \psi, M \nabla \psi \rangle \rmd \mu_0^x(v) \rmd x  \\
= &\int_{\R^d} \bar{u}_{0}(x) W_x^2(\mu_0^x, \mu_1^x) \rmd x \\
= &W_J^2(\mu_0, \mu_1).
\end{aligned}
\end{equation}

By the definition of the partially relative Fisher information in \eqref{fisher}, we obtain
\begin{equation}\label{of}
\begin{aligned}
I(\mu_0 | \mu_{\infty})= \int_{\R^{2d}} \big\langle \nabla_v \theta( u_0(x,v)), M \nabla_v \theta( u_0(x,v)) \big\rangle u_0(x,v) \rmd x\rmd v.
\end{aligned}
\end{equation}

Combining \eqref{taylor}, \eqref{om} and \eqref{of} completes the proof of \eqref{HWI}.
\end{proof}

\begin{rem}
(i) Based on \eqref{fenergy} and \eqref{coraver}, we obtain
\begin{equation*}
\begin{aligned}
H(\mu_t | \mu_{\infty})&=\tilde{\F}_t(\mu_t)=\F(\rho_t)=\int_{\R^{2d}} f_t \ln f_t + \beta m f_t \big( \dfrac{1}{2}v^2+U(x) \big) \rmd z, \\
I(\mu_t | \mu_{\infty})&=\tilde{\I}(\mu_t)=\I(\rho_t)=\dfrac{1}{2}\sigma^2 \int_{\R^{2d}} | \nabla_v (\ln f_t+ \beta m v ) |^2 f_t \rmd x \rmd v
\end{aligned}
\end{equation*}
for $\rmd \mu_{\infty}(\tilde{z})=\rmd \rho_{\infty}(z):=\rme^{-\beta m (\frac{1}{2}v^2+U(x))}\rmd z$.
Therefore, \eqref{HWI} can also be written as
\begin{equation*}
\tilde{\F}_0(\mu_0)-\tilde{\F}_1(\mu_1)\leq \sqrt{\tilde{\I}(\mu_0)}~ W_J(\mu_0, \mu_1) - \dfrac{\kappa_M}{2}W_J^2(\mu_0, \mu_1).
\end{equation*}
(ii)
The favorable fiberwise property, namely that the metric reduces to a classical
$W_2$-metric on a Riemannian manifold, relies crucially on the condition $\nabla K=0$.
In the nonlinear case, the analysis becomes substantially more involved. In particular, one can no longer define the optimal transport map $T^x$ in a fiberwise manner.
\end{rem}

\section*{Acknowledgements}
We are grateful to Professor Nicola Gigli for insightful comments.
This work is supported by National Key R\&D Program of China (No. 2023YFA1009200), NSFC (Grants 12531009 and 11925102), and Liaoning Revitalization Talents Program (Grant XLYC2202042).


\begin{thebibliography}{xx}

\bibitem{AG}
L. Ambrosio, N. Gigli,
\newblock A user's guide to optimal transport.
\newblock {\em Modelling and Optimisation of Flows on Networks, 2013, 1--155.} Springer, Heidelberg; Fondazione C.I.M.E., Florence, 2013. xiv+497 pp.

\bibitem{AGS}
L. Ambrosio, N. Gigli, G. Savar\'e,
\newblock {\em Gradient Flows in Metric Spaces and in the Space of Probability Measures.}
\newblock Second edition. Birkh\"auser Verlag, Basel, 2008. x+334 pp.

\bibitem{BLPR}
R. Buckdahn, J. Li, S. Peng, C. Rainer,
\newblock Mean-field stochastic differential equations and associated PDEs.
\newblock {\em Ann. Probab.} \textbf{45} (2017), 824--878.

\bibitem{C}
R. Carmona,
\newblock {\em Lectures on BSDEs, Stochastic Control, and Stochastic Differential Games with Financial Applications.}
\newblock SIAM, Philadelphia, PA, 2016. ix+265 pp.

\bibitem{CD}
R. Carmona, F. Delarue,
\newblock Probabilistic theory of mean field games with applications. I. Mean field FBSDEs, control, and games.
\newblock {\em Probab. Theory Stoch. Model., 83.} Springer, Cham, 2018. xxv+713 pp.

\bibitem{DPZ}
M. Duong, M. A. Peletier, J. Zimmer,
\newblock GENERIC formalism of a Vlasov-Fokker-Planck equation and connection to large-deviation principles.
\newblock {\em Nonlinearity} \textbf{26} (2013), 2951--2971.

\bibitem{FF}
A. Fathi, A. Figalli,
\newblock Optimal transportation on non-compact manifolds.
\newblock {\em Isr. J. Math.} \textbf{175} (2010), 1--59.

\bibitem{JKO2}
R. Jordan, D. Kinderlehrer, F. Otto,
\newblock The variational formulation of the Fokker-Planck equation.
\newblock {\em SIAM J. Math. Anal.} \textbf{29} (1998), 1--17.

\bibitem{K}
M. Kac,
\newblock Foundations of kinetic theory.
\newblock {\em Proceedings of the Third Berkeley Symposium on Mathematical Statistics and Probability, 1954--1955, vol. III, 171--197.} University of California Press, Berkeley-Los Angeles, Calif., 1956.

\bibitem{KC}
D. Kim, L. Chun Yeung,
\newblock A trajectorial approach to entropy dissipation for degenerate parabolic equations.
\newblock {\em Bernoulli} \textbf{30} (2024), 2253--2274.

\bibitem{KLMP}
R. Kraaij, A. Lazarescu, C. Maes, M. A. Peletier,
\newblock Deriving GENERIC from a generalized fluctuation symmetry.
\newblock {\em J. Stat. Phys.} \textbf{170} (2018), 492--508.

\bibitem{KRTY}
A. Kazeykina, Z. Ren, X. Tan, J. Yang,
\newblock Ergodicity of the underdamped mean-field Langevin dynamics.
\newblock {\em Ann. Appl. Probab.} \textbf{34} (2024), 3181--3226.

\bibitem{KST}
I. Karatzas, W. Schachermayer, B. Tschiderer,
\newblock A trajectorial approach to the gradient flow properties of Langevin-Smoluchowski diffusions.
\newblock {\em Theory Probab. Appl.} \textbf{66} (2022), 668--707.

\bibitem{LW}
Z. Liu, X. Wang,
\newblock A trajectorial approach to the gradient flow of McKean-Vlasov SDEs with mobility, arXiv: 2501.11913 [math.PR].

\bibitem{MR}
J. E. Marsden, T. S. Ratiu,
\newblock {\em Introduction to Mechanics and Symmetry.}
\newblock Second edition. Springer-Verlag, New York, 1999. xviii+535 pp.

\bibitem{MPZ}
A, Mielke, M. A. Peletier, J. Zimmer,
\newblock Deriving a GENERIC system from a Hamiltonian system.
\newblock {\em Arch. Ration. Mech. Anal.} \textbf{249} (2025), Paper No. 62, 71 pp.

\bibitem{OTTO}
F. Otto,
\newblock The geometry of dissipative evolution equations: the porous medium equation.
\newblock {\em Comm. Partial Differential Equations} \textbf{26} (2001), 101--174.

\bibitem{S2}
K. Sturm,
\newblock Convex functionals of probability measures and nonlinear diffusions on manifolds.
\newblock {\em J. Math. Pures Appl.} \textbf{84} (2005), 149--168.

\bibitem{TC}
B. Tschiderer, L. Chun Yeung,
\newblock A trajectorial approach to relative entropy dissipation of McKean-Vlasov diffusions: gradient flows and HWBI inequalities.
\newblock {\em Bernoulli} \textbf{29} (2023), 725--756.




\end{thebibliography}

\end{document}